\documentclass[11pt]{article}
\usepackage{fullpage}
\usepackage{authblk}

\usepackage{amsmath}
\usepackage{amsthm}
\usepackage{amssymb}
\usepackage{graphics}
\usepackage{graphicx}
\usepackage{bm}
\usepackage{setspace}

\usepackage{amsfonts}
\usepackage{eucal}
\usepackage{latexsym}
\usepackage{mathdots}
\usepackage{mathrsfs}
\usepackage{enumitem}
\usepackage{color}

\newcommand{\be}{\begin{equation}}
\newcommand{\ee}{\end{equation}}
\newcommand{\ba}{\begin{eqnarray}}
\newcommand{\ea}{\end{eqnarray}}
\newcommand{\bal}{\begin{align}}
\newcommand{\eal}{\end{align}}
\newcommand{\baln}{\begin{align*}}
\newcommand{\ealn}{\end{align*}}
\newcommand{\bi}{\begin{itemize}}
\newcommand{\ei}{\end{itemize}}
\newcommand{\bn}{\begin{enumerate}}
\newcommand{\en}{\end{enumerate}}
\newcommand{\bbm}{\begin{bmatrix}}
\newcommand{\ebm}{\end{bmatrix}}
\newcommand{\bpm}{\begin{pmatrix}}
\newcommand{\epm}{\end{pmatrix}}
\newcommand{\bp}{\begin{proof}}
\newcommand{\ep}{\end{proof}}
\newcommand{\nn}{\nonumber}

\newcommand{\mr}{\ensuremath{\mathrm}}
\newcommand{\scr}{\ensuremath{\mathscr}}

\newcommand{\mc}{\ensuremath{\mathcal}}
\newcommand{\mf}{\ensuremath{\mathfrak}}
\newcommand{\ov}{\ensuremath{\overline}}
\newcommand{\sm}{\ensuremath{\setminus}}
\newcommand{\wt}{\ensuremath{\widetilde}}

\newcommand{\ga}{\ensuremath{\gamma}}
\newcommand{\Om}{\ensuremath{\Omega}}

\newcommand{\la}{\ensuremath{\lambda }}

\newcommand{\eps}{\ensuremath{\epsilon }}

\def\C{\mathbb{C}}
\def\R{\mathbb{R}}
\def\D{\mathbb{D}}

\def\N{\mathbb{N}}
\def\B{\mathbb{B}}

\def\A{\mathcal{A} _d}

\def\fp{\mathbb{C} \{ \mathfrak{z} \} }
\def\hardy{\mathbb{H} ^2 _d}
\def\mult{\mathbb{H} ^\infty _d}
\def\bH{\mathbb{H}}
\def\mrt{\mathrm{t}}
\def\nbdom{\mr{Dom} \, }
\def\nbran{\mr{Ran} \, }
\def\nbre{\mr{Re} \,}
\def\nbim{\mr{Im} \,}
\def\nbker{\mr{Ker} \,}

\def\posncm{ \left( \scr{A} _d   \right) ^\dag _+ }
\newcommand{\bsm}{\left ( \begin{smallmatrix}}
\newcommand{\esm}{\end{smallmatrix} \right) }

\def\cH{\mathcal{H}}

\newcommand{\verteq}{\rotatebox{90}{$\,=$}}

\newcommand{\F}{\ensuremath{\mathbb{F} }}

\newcommand{\inv}[1]{\rotatebox[origin=c]{180}{#1}}

\newcommand{\nt}{\ensuremath{\stackrel{\leftarrow}{\scriptstyle \sphericalangle}}}

\newcommand{\ip}[2]{\ensuremath{\langle {#1} , {#2} \rangle}}
\newcommand{\ipcn}[2]{\ensuremath{\left( {#1} , {#2} \right) _{\C ^n}}}

\newcommand{\dom}[1]{\ensuremath{\mathrm{Dom} \left( {#1} \right) }}

\newcommand{\ran}[1]{\ensuremath{\mathrm{Ran} \left( {#1} \right) }}

\renewcommand{\ker}[1]{\ensuremath{\mathrm{Ker} ({#1}) }}

\numberwithin{equation}{section}
\numberwithin{subsection}{section}

\newtheorem{thm}[subsection]{Theorem}

\newtheorem{lemma}[subsection]{Lemma}
\newtheorem{prop}[subsection]{Proposition}
\newtheorem{cor}[subsection]{Corollary}
\newtheorem*{thm*}{Theorem}

\theoremstyle{definition}
\newtheorem{defn}[subsection]{Definition}
\newtheorem{remark}[subsection]{Remark}
\newtheorem{eg}[subsection]{Example}

\begin{document}
\title{Fatou's Theorem for Non-commutative Measures}
\date{}

\author{Michael T. Jury\thanks{Supported by NSF grant DMS-1900364}}
\affil[1]{ University of Florida}

\author[2]{Robert T.W. Martin\thanks{Supported by NSERC grant 2020-05683}}
\affil[2]{ University of Manitoba}

\bibliographystyle{unsrt}
\maketitle
\onehalfspace

\begin{abstract}
A classical theorem of Fatou asserts that the Radon-Nikodym derivative of any finite positive Borel measure, $\mu$, with respect to Lebesgue measure on the complex unit circle, is recovered as the non-tangential limits of its Poisson transform in the complex unit disk. This positive harmonic Poisson transform is the real part of an analytic function whose Taylor coefficients are in fixed proportion to the conjugate moments of $\mu$.

Replacing Taylor series in one variable by power series in several non-commuting variables, we show that  Fatou's Theorem and related results have natural extensions to the setting of positive harmonic functions in an open unit ball of several non-commuting matrix-variables, and a corresponding class of positive \emph{non-commutative (NC) measures}. Here, an NC measure is any positive linear functional on a certain self-adjoint unital subspace of the Cuntz-Toeplitz algebra, the $C^*-$algebra generated by the left creation operators on the full Fock space. 
\end{abstract}

\section{Introduction}

The goal of this paper is to extend results from classical measure theory and the theory of Hardy Spaces of analytic functions in the open unit disk, $\mathbb D$, in the complex plane, from one to several non-commuting variables. In particular we are interested in Fatou's theorem \cite{Fatou}, which recovers the Radon-Nikodym derivative of a positive measure $\mu$ on the unit circle (with respect to Lebesgue measure) from the boundary values of its Poisson integral transform. (This positive harmonic Poisson transform is the real part of an analytic function $h_\mu$, given by the Herglotz-Riesz integral of $\mu$.)  In the non-commutative (NC) setting, positive measures on the circle are replaced with positive linear functionals on the \emph{free disk system} (these will be called `NC measures'; all terminology will be defined carefully below), the NC multi-variable analogue of the operator system of the disk algebra. (Here, recall that the disk algebra is the unital Banach algebra of analytic functions in $\D$ which extend continuously to the boundary, equipped with the supremum norm.) The notion of Herglotz-Riesz integral transform has a natural extension to this non-commutative setting, and the NC Herglotz-Riesz transform of any positive NC measure, $\mu$, is an NC Herglotz function, $H_\mu$, with positive real part in a certain non-commutative, multi-variable unit row-ball.  We are then faced with two basic problems: First, develop a suitable definition of the Lebesgue decomposition $\mu=\mu_{ac}+\mu_s$ of an NC measure (with respect to the `vacuum state', which is the canonical analogue of Lebesgue measure), and secondly, assuming such an NC Lebesgue decomposition theory can be developed, find a method of recovering the absolutely continuous part, $\mu _{ac}$, of the NC measure $\mu$ from its NC Herglotz-Riesz transform $H_\mu$. We will put forward solutions to both of these problems. 

Let us describe the background in more detail; this will also allow us to establish some notation. We recall that the Hardy Space, $H^2  = H^2 (\D )$, is the Hilbert space of analytic functions in the unit disk with square-summable Taylor coefficients at $0 \in \D$ (endowed with the $\ell ^2$ inner product of these coefficients). Any element of the Hardy space has non-tangential boundary limits almost everywhere with respect to Lebesgue measure on the unit circle in the complex plane, $\partial \D$, and the identification of $h \in H^2$ with its boundary limits is an isometry into $L^2 (\partial \D )$.  The Hardy algebra, $H^\infty = H^\infty (\D )$, is the unital Banach algebra of bounded analytic functions in $\D$. Multiplication by any $h \in H^\infty$ defines a bounded linear map from $H^2$ into itself, so that $H^\infty$ can be viewed as the \emph{multiplier algebra} of $H^2 (\D )$. 

In measure theory and Hardy space theory, there is an (essentially) bijective correspondence between finite, positive, and regular Borel measures on the unit circle $\partial \D$, and contractive analytic functions in $\D$. Namely, 
if $b \in [H^\infty ]_1$ is a contractive analytic function in $\D$, then its Cayley Transform, 
$$ h_b := \frac{1+b}{1-b}, $$ is a \emph{Herglotz function}, an analytic function with positive semi-definite real part in $\D$. By the Herglotz-Riesz representation formula for positive harmonic functions, there is a unique positive measure, $\mu _b$, so that the Herglotz-Riesz transform of $\mu _b$ is 
$$ h_{\mu _b} (z)  :=  \int _{\partial \D} \frac{1 + z \zeta ^*}{1-z\zeta ^*} \mu _b (d\zeta) = h_b (z) - i \nbim h_b (0). $$ (In the above, and throughout, $c^* = \ov{c}$ denotes complex conjugate of a complex number $c \in \C$.) This measure, $\mu_b$, is called the Clark or Aleksandrov-Clark measure of $b \in [H^\infty ] _1$, and many properties of $b$ are reflected in those of $\mu _b$ \cite{Clark1972,Aleks2,Aleks1,Saks}. For example, a theorem of Fatou, \cite{Fatou} (see \cite[Chapter 3.3: Fatou's Theorem]{Hoff}), implies that the Radon-Nikodym derivative of $\mu _b$ with respect to normalized Lebesgue measure, $m$, on $\partial \D$ is given ($m -a.e.$) by the non-tangential (and in particular, radial) limits of the positive harmonic function $\nbre h_b$:
\ba \frac{\mu _b (d\zeta)}{m(d\zeta)} & =  & \lim _{z\inv{\nt} \zeta} \nbre \frac{1 + b(z)}{1-b(z)} \quad \quad (m-a.e.) \nn \\
& = & \lim _{z\inv{\nt} \zeta} (1 - b(z ) ^* ) ^{-1} (1 - b(z ) ^* b(z) ) (1 - b(z) ) ^{-1}; \quad m-a.e. \nn \\
& = & \lim _{z \inv{\nt} \zeta} \frac{1-|b(z)|^2}{|1-b(z)| ^2} \quad \quad (m-a.e.), \nn \ea where $\zeta \in \partial \D$, $z \in \D$, and $z \ \inv{\nt} \ \zeta$ denotes non-tangential convergence.  In particular, it follows that the Clark measure, $\mu _b$, of $b$ is singular with respect to Lebesgue measure if and only if $b$ is \emph{inner}, \emph{i.e.} has unimodular non-tangential limits $m-$a.e. on $\partial \D$. (Under the identification of $h \in H ^2$ with its boundary values, it is clear that any contractive analytic $b \in [ H^\infty ] _1$ is inner if and only if the multiplier $M_b$ is an isometry of $H^2$ into itself.) 

This bijective correspondence between positive measures and contractive analytic functions extends naturally to the non-commutative multi-variable setting of the \emph{Non-commutative} (NC) or \emph{free Hardy space}, \cite[Section 5]{Pop-freeharm} (see also \cite{JM-freeAC,JM-freeCE}). Here, the free Hardy space can be viewed as a Hilbert space of (graded) analytic functions in an open unit ball of $d-$tuples of matrices (of all sizes) taking values in matrices of all sizes \cite{JM-freeCE,Pop-freeholo,Pop-freeholo2,BMV,VinPop,KVV}. Elements of this NC Hardy space have Taylor series expansions indexed by the free monoid, $\F ^d$, the set of all words in $d$ letters. It follows that the free Hardy space is isomorphic to $\ell ^2 (\F ^d )$ in the same way that $H^2 (\D )$ is isomorphic to the square-summable sequences indexed by the non-negative integers, $\ell ^2 (\N _0 )$. 

By the Riesz-Markov representation theorem, any finite, positive, and regular Borel measure, $\mu$, on $\partial \D$, can be identified with a positive linear functional, $\hat{\mu }$ on $\scr{C} ( \partial \D )$, the commutative $C^*-$algebra of continuous functions on the circle:
$$ \hat{\mu} (f) := \int _{\partial \D } f (\zeta ) d\mu(\zeta ). $$ By the Weierstrass approximation theorem, 
$$ \scr{C} (\partial \D ) = \left( \mc{A} (\D ) + \mc{A} (\D ) ^* \right) ^{-\| \cdot \| _\infty}, $$ where $\mc{A} (\D )$ is the \emph{disk algebra}, the algebra of all analytic functions in $\D$ with continuous extensions to the boundary. In the above formula, elements of $\mc{A} (\D )$ are identified with their continuous boundary values and $\| \cdot \| _\infty$ denotes the supremum norm for continuous functions on the circle. The disk algebra can also be viewed as the norm-closed unital operator algebra generated by the \emph{shift}, $S := M_z$, the canonical isometry of multiplication by $z$ on the Hardy space, $H^2 (\D )$. The shift plays a central role in the theory of Hardy spaces \cite{Nik-shift,NF}.  The positive linear functional $\hat{\mu}$ is then completely determined by the moments of the measure $\mu$: 
\be \hat{\mu}  (S^k ) := \int _{\partial \D} \zeta ^k d\mu(\zeta ), \label{RHtrans} \ee and the Taylor series coefficients (at $0$) of the Herglotz-Riesz integral transform of $\mu$ are $$ h_\mu (0) = \hat{\mu} (I), \quad \quad h_{\mu ;k} := \frac{1}{k!} \frac{d ^k h_\mu}{dz ^k}  (0) =  2 \ov{\hat{\mu} (S^k)}, \quad k \geq 1. $$

The shift on $H^2 (\D )$ is isomorphic to the unilateral shift on $\ell ^2 (\N _0 )$. The square-summable sequences, $\ell ^2 (\N _0 )$, can in turn be viewed as a simple, directed tree starting from a single node and with one branch directed downward from each node to the next.  A canonical several-variable extension of $\ell ^2 (\N _0 )$ is then $\ell ^2 (\F ^d )$. If we view, as before, $\ell ^2$ of the free monoid as a simple directed tree starting from a single node and with $d$ branches directed downwards from each node, it is natural to define a $d-$tuple of isometries, the \emph{left free shifts}, $L_k$, $1\leq k \leq d$ which shift along these branches from nodes indexed by words of length $N$ to those of length $N+1$ (by appending letters to the left of words indexing the standard orthonormal basis). These left free shifts have pairwise orthogonal ranges so that the row operator $L := (L_1, \cdots , L_d ) : \ell ^2 (\F ^d ) \otimes \C ^d \rightarrow \ell ^2 (\F ^d)$ is an isometry from $d$ copies of $\ell ^2 (\F ^d )$ into one copy which we call the \emph{left free shift}. This space of square-summable free sequences, $\ell ^2 (\F ^d)$, can also be identified with the full Fock space over $\C ^d$, the direct sum of all tensor powers of $\C ^d$. The full Fock space will be denoted by $\bH ^2 _d$.  Under this isomorphism the left free shifts are conjugate to the left creation operators, see Section \ref{back} for more details. 

The immediate analogue of a positive measure in this non-commutative (NC) multi-variable setting is then a positive linear functional, or \emph{NC measure}, on the \emph{free disk system}:
$$ \scr{A} _d := \left( \A  + \A  ^* \right) ^{-\| \cdot \| }, $$ where $\A := \mr{Alg} \{ I , L_1 , \cdots , L_d \} ^{-\| \cdot \|}$ is the \emph{free disk algebra}, the operator norm-closed unital operator algebra generated by the left free shifts. As in the classical theory, elements of the free disk algebra can be identified with bounded (matrix-valued) analytic functions (in several non-commuting matrix variables) which extend continuously from the interior to the boundary of a certain non-commutative multi-variable open unit ball. Moreover, exactly as in Equation (\ref{RHtrans}), the NC moments of the NC measure, $\mu$, are in fixed proportion to the Taylor series coefficients of its NC Herglotz-Riesz transform, $H_\mu$. There is a long history and well-established precedent of viewing positive linear functionals on operator algebras as non-commutative measures, \emph{e.g.} in von Neumann algebra theory \cite{Kosaki,Takesaki}. (Indeed, von Neumann algebra theory has been called non-commutative measure theory \cite{Connes}.) 

Any NC measure can be viewed as a positive semi-definite quadratic form with dense domain $\mc{A} _d \subset \bH ^2 _d$, and we develop a Lebesgue decomposition theory of (positive) NC measures into absolutely continuous and singular parts using the Lebesgue decomposition theory for quadratic (\emph{i.e.} sesquilinear) forms due to B. Simon \cite[Section 2]{Simon1}, combined with (non-commutative) reproducing kernel techniques applied to the spaces of Cauchy transforms with respect to the NC measure. This Lebesgue decomposition theorem for (potentially unbounded) quadratic forms in Hilbert Space is similar in spirit to von Neumann's proof of the Radon-Nikodym Theorem and Lebesgue decomposition theory \cite[Lemma 3.2.3]{vN3}, and our computation of the Radon-Nikodym derivative of any positive NC measure is also reminiscent of this approach. In our theory, the Radon-Nikodym derivative of a non-commutative (NC) measure with respect to NC Lebesgue measure will be a (generally unbounded) positive semi-definite left Toeplitz operator in the sense of \cite{Pop-entropy,Pop-freeharm} and defines a positive free pluriharmonic function in the sense of Popescu \cite{Pop-freeharm}.

\begin{remark}[On boundary values]
The free Hardy space, $H^2 (\B ^d _\N )$, is a Hilbert space of non-commutative (NC) functions in an open unit ball, $\B ^d _\N$, of the \emph{NC universe},
\be \C ^d _\N := \bigsqcup _{n=1} ^\infty  \C ^{(n\times n ) \cdot d}; \quad \quad \C ^{(n\times n) \cdot d} := \C ^{n\times n} \otimes \C ^{1\times d}. \label{NCuni} \ee In the above, $\C ^{m \times n}$ denotes the $m \times n$ matrices with complex entries, and any element of $\C ^{n \times n} \otimes \C ^{1\times d}$ is viewed as a row $d-$tuple, $Z = (Z_1 , \cdots , Z_d )$, of $n \times n$ matrices. As described in \cite{VinPop}, one can equip $\C ^d$ with three natural operator space norms: Given $X \in \C ^d _\N$, 
$$ \| X \| _{\infty} := \mr{max} \{ \| X _1 \| , \cdots , \| X_d \| \},   \quad  \| X \| _{col}  := \left\| \sum _{j=1} ^d X_j ^* X_j \right\| ^{\frac{1}{2}}, \quad 
  \mbox{and} \quad  \| X \| _{row}  := \left\| \sum _{j=1} ^d X_j X_j ^* \right\| ^{\frac{1}{2}}. $$ 
The unit ball we consider, is the unit ball, $\B ^d _\N$, of $\C ^d _\N$ with respect to the row-norm, $\| \cdot \| _{row}$. That is, we consider the \emph{NC open unit row-ball}, 
\be \B ^d _\N = \bigsqcup _{n=1} ^\infty \B ^d _n; \quad \quad \B ^d _n := \left\{ \left. Z = (Z_1, \cdots, Z_d ) \in \C ^{(n\times n) \cdot d} \ \right| \ Z_1 Z_1 ^* + \cdots + Z_d Z_d ^* < I_n \right\}. \label{NCball} \ee The open unit ball of $\C ^d _\N$ with respect to the $\| \cdot \| _\infty$ operator space norm is the \emph{NC unit polydisk}, $\D ^d _\N$, consisting of all points $X = (X_1, \cdots , X_d ) \in \C ^d _\N$ so that $\| X _k \| <1$, for $1\leq k \leq d$.

The distinguished boundary of the NC unit row ball, $\partial \B ^d _\N$, can be identified with the set of all row co-isometries, \emph{i.e.} the set of all $X \in \C ^{(n\times n) \cdot d}$ obeying $X X ^*= X_1 X_1 ^* + \cdots + X_d X_d ^* = I_n$. (Note that there are no finite-dimensional row isometries.) Indeed, if we identify the components of $X \in \B ^d _n$, with the corresponding subset of $\C ^{n^2 \cdot d}$, then the Shilov boundary for the commutative algebra of complex functions in $n^2 d$ variables is the set of all co-isometries in $\C ^{n \times d n}$ \cite[Example 1.5.51]{Upbook}.  Moreover, as proven in \cite[Lemma 2.10]{SSS}, any (irreducible) point $X \in \C ^{(n\times n) \cdot d}$ with joint spectral radius equal to $1$, which is not in the joint similarity envelope of the NC unit ball, is jointly similar to a row co-isometry. On the other hand, as described in Subsection \ref{rkhs} below, if $X \in \C ^{(n\times n)\cdot d}$ is in the joint similarity envelope of the NC unit row ball then the point evaluation map $h \mapsto h(X)$ is bounded as a linear map of $H^2 (\B ^d _\N )$ into the Hilbert space $\left( \C ^{n\times n} , \mr{tr} _n \right)$.

It is natural to consider the boundary values of $h \in H^2 (\B ^d _\N )$ on the boundary of the NC unit row-ball, as well as to wonder whether the inner product in the NC Hardy space can be expressed as an integral over the boundary $\partial \B ^ d_\N$. Moreover, one may wonder if there is an exact analogue of Fatou's theorem in this context. Investigations of boundary values in NC Hardy spaces of functions in the NC unit polydisk (and in the NC unit balls with respect to the $\| \cdot \| _{col}$ and $\| \cdot \| _{row}-$norms) were studied in \cite{VinPop}. See also \cite[Chapters 14-16]{Voic2} which studies asymptotic tracial integral formulae of bounded non-commutative functions in $\D ^d _\N$ over the distinguished boundary, $\partial \D ^d _\N$ of $\D ^d _\N$, consisting of $d-$tuples of unitary matrices. In particular, \cite[Theorem 3.5]{VinPop} shows that one can construct an NC Hardy space of NC functions, $H^2 (\D ^d _\N )$ in $\D ^d _\N$ with inner product defined as a limit of tracial integrals with respect to product Haar measure over the boundary of $r \partial \D ^d _\N$, for $0<r<1$, and this recovers the Fock space inner product at least on NC polynomials $p,q \in H^2 (\D ^d _\N )$. Note, however, that $\B ^d _\N \subsetneq \D ^d _\N$, the NC unit row-ball is a proper subset of the NC unit polydisk, and a general element of $H^2 (\B ^d _\N)$ need not (and generally does not) extend to $\D ^d _\N$. (It is not difficult to construct examples of NC rational functions $\mathfrak{r}$ in $\hardy$ for which certain points $Z \in \D ^d _\N \sm \B ^d _\N$ do not belong to their domains, for example see \cite{JMS-NCrat}.)  Thus, it appears that generic elements of the Fock space cannot be represented as NC functions in the polydisk $\D ^d _\N$, and in particular there does not appear to be any way to sensibly assign levelwise boundary values on $\partial\D ^d _\N$ to arbitrary elements of the Fock space. Moreover, while the Fock space inner product of {\em polynomials} can be given a random matrix interpretation, this interpretation does {\em not} extend to all elements of the Hilbert space. On the other hand, \cite[Theorem 3.5]{VinPop} can also be applied to construct a Hardy space of NC functions in the row-ball $\B ^d _\N$ whose inner product is given by an asymptotic integral formula with respect to a family of invariant measures over each level of the boundary. However, this inner product does not coincide with the Fock space inner product and hence this gives a different NC Hardy space than the one we consider here.

At this point, it is not obvious whether or not the constructions of \cite{VinPop} can be modified in a suitable way to show that the inner product in $H^2 (\B ^d _\N)$, can be expressed in terms of integral formulae over the boundary, $\partial \B ^d _\N$. This would seem to be a necessary first step in developing an exact analogue of Fatou's theorem in this setting. While such an approach would be interesting and valuable, we will instead pursue a more abstract and operator-theoretic approach. By re-casting Fatou's theorem in purely operator theoretic language, we will develop a `dimension-free' proof of Fatou's theorem that applies to $H^2 (\B ^d _\N )$, independently of $d \in \N$. In particular, we recover the classical result for $H^2 (\D ) \simeq H^2 (\B ^1 _\N )$ as a special case. 
\end{remark}

\subsection{Readers' guide} Section~\ref{back} provides the necessary backgound material on the free disk algebra, NC function theory, and the formalism of non-commutative reproducing kernels. Section~\ref{NCmeasuresect} recalls the relevant background on NC measures and the NC Cauchy transform from \cite{JM-freeCE,JM-freeAC}. In Section~\ref{sec:simon} we describe Simon's approach to the Lebesgue decomposition of positive semi-definite (not necessarily bounded) quadratic forms in Hilbert space, see Theorem \ref{maxclosform}. We use this theory to define the Lebesgue decomposition of an NC measure into absolutely continuous and singular parts in Definition \ref{NCncld}; however at this stage it is not clear that the absolutely continuous and singular parts of NC measures are themselves NC measures. Section~\ref{FreeCTsect} then proves that this quadratic form approach produces NC measures by examining spaces of NC Cauchy transforms. This yields an NC Lebesgue decomposition of the original NC measure as the sum of absolutely continuous and singular (positive) NC measures in Theorem \ref{NCLD}. We further show that the NC Radon-Nikodym derivative, $T$, of any positive NC measure with respect to NC Lebesgue measure is a closed, positive semi-definite operator with the \emph{$L-$Toeplitz property}:
$$ \ip{\sqrt{T} L_k h}{\sqrt{T} L_j g}_{\bH ^2 _d} = \delta _{k,j} \ip{\sqrt{T} h}{\sqrt{T} g} _{\bH ^2 _d}; \quad \quad h,g \in \mr{Dom} \sqrt{T}, $$ in Theorem \ref{ACiffclosed}.  
Corollaries \ref{ACcor} and \ref{singcor} provide further characterizations of absolutely continuous and singular NC measures in terms of their Gelfand-Naimark-Segal spaces and their spaces of NC Cauchy transforms. Finally in Section~\ref{NCRN} we prove our NC Fatou theorem, Theorem \ref{main}, which shows how to recover the absolutely continuous part of an NC measure from its Herglotz-Riesz transform; the main tool is the theory of strong resolvent convergence of self-adjoint operators:
\begin{thm*}[Non-commutative Fatou Theorem] 
Let $\mu \in \posncm$ be any positive NC measure with NC Herglotz-Riesz transform $H_\mu$. The absolutely continuous part of $\mu$ is given by the formula:
$$ \mu _{ac} (a_1 ^* a_2 ) = \ip{\sqrt{T} a_1}{\sqrt{T} a_2}_{\bH ^2 _d}; \quad \quad a_1, a_2 \in \A, $$ where the NC Radon-Nikodym derivative, $T$, is a closed, positive semi-definite $L-$Toeplitz operator with dense domain in $\bH^2 _d$ so that $\mc{A} _d$ is a core for $\sqrt{T}$. This NC Radon-Nikodym derivative can be computed by the formula: 
$$ (T + \eps I ) ^{-1} = SOT-\lim _{r\uparrow 1} \left( \nbre H_\mu (rR) + \eps I \right) ^{-1}; \quad \quad \eps >0. $$
\end{thm*}
In the above, and throughout, $SOT$ denotes the strong operator topology and $R = (R_1, \cdots , R_d)$ is the row isometric \emph{right free shift} whose components, $R_k$, act as $R_k e _\alpha = e_{\alpha k}$ on the standard orthonormal basis, $\{ e_\alpha \} _{\alpha \in \F ^d }$ of $\ell ^2 (\F ^d )$. As a corollary to this main result, we recover (half of) a familiar fact from Hardy space theory in Corollary \ref{innercor} -- if $B$ is an NC inner function in $\B ^d _\N$, then its NC Clark measure is singular with respect to NC Lebesgue measure.

\section{Background: The free Hardy space} \label{back}

The Hardy space, $H^2 (\D )$, of analytic functions in the complex open unit disk can be defined in two equivalent ways: On one hand, 
$$ H^2 (\D ) := \left\{ \left. f(z) = \sum _{k=0} ^\infty \hat{f} _n z^n  \in \scr{O} ( \D ) \right| \ \sum |\hat{f} _n | ^2 < \infty  \right\} $$ is the Hilbert space of all analytic functions in $\D$ with square-summable MacLaurin series coefficients (and with the $\ell ^2-$inner product of these coefficients). Alternatively, 
$$ H^2 (\D ) := \cH (k); \quad \quad k(z,w) := \frac{1}{1-zw^*}; \quad z,w \in \D, $$ is the unique reproducing kernel Hilbert space (RKHS) corresponding to the positive sesqui-analytic Szeg\"{o} kernel, $k$, on $\D$. (Here, recall that there is a bijective correspondence between reproducing kernel Hilbert spaces of complex-valued functions on a set $X$, and positive kernel functions on $X \times X$. See Subsection \ref{rkhs} for a detailed description.)

One also has the corresponding two equivalent definitions of the Hardy algebra $H^\infty (\D )$:
$$ H^\infty (\D ) := \left\{ h \in \scr{O} (\D ) \ \left| \ \sup _{z \in \D } | h(z ) | < \infty \right. \right\}, $$ and
$$ H^\infty (\D ) := \mr{Mult} \, \cH (k) ; \quad \quad k (z,w) = (1-zw^*) ^{-1}, $$ where recall that the multiplier algebra, $\mr{Mult} \, \cH (k) $, of an RKHS, $\cH (k)$, is the unital $WOT-$closed (weak operator topology closed) algebra of all functions $h$ which multiply $\cH (k)$ into itself. These definitions have natural non-commutative multi-variable extensions.

\subsection{ $\ell ^2$ of the free monoid as the free Hardy space}

 From the square-summable Taylor series definition, is clear that the shift, $S = M_z$, is an isometry on $H^2 (\D )$.  The map $f \mapsto (\hat{f} _n ) $ of a holomorphic $f \in H^2 (\D )$ to its sequence of MacLaurin series coefficients defines a unitary map from $H^2 (\D )$ onto $\ell ^2 (\N _0 )$, the square-summable sequences indexed by the non-negative integers, and the image of the shift under this unitary transformation is $\hat{S}$, the unilateral shift on $\ell ^2 (\N _0)$, the \emph{universal pure isometry}. Here, recall that the Wold decomposition shows that any isometry on Hilbert space is unitarily equivalent to several copies of the unilateral shift, with possibly a remainder unitary direct summand.  A canonical multi-variable extension of $\ell ^2 (\N _0 )$ is then $\ell ^2 (\F ^d) $, the square-summable sequences indexed by $\F ^d$, the free and universal monoid on $d$ letters. This is the unital semi-group of all words in $d$ letters (which we choose to be $\{ 1, ... ,d \}$), with product given by concatenation of words, and unit equal to the empty word, $\emptyset$, containing no letters (and clearly $\N _0 \simeq \F ^1$). Given any word, $\alpha = i_1 \cdots i_n \in \F ^d$, $i_k \in \{ 1 , \cdots , d \}$, we write $|\alpha | := n$ for the length of the word $\alpha$, and $|\emptyset | =0$. The square-summable sequences $\ell ^2 (\N _0 )$ can be viewed as a directed graph or tree, starting from a single node (labelled by $0$, the unit of $\N _0$), and with one branch connecting each node to the next. The unilateral shift moves downward along these branches, taking the orthonormal basis vector of the $k^{\mr{th}}$ node to that of $k+1$. Similarly, $\ell ^2 (\F ^d)$ can be viewed as a tree starting from a single node labelled by the unit, $\emptyset$, and with $d$ branches directed downward from each node.  In this multi-variable setting there is now a natural $d-$tuple of \emph{left free shifts} shifting along branches from a given node to $d$ distinct nodes at the next level. Namely, if $e _\alpha$ is the orthonormal basis vector labelled by the word $\alpha \in \F ^d$, then 
$L_k e _\alpha = e_{k\alpha}; \ 1\leq k \leq d$. It is not difficult to see that the left free shifts are isometries with pairwise orthogonal ranges, 
$$ L_k ^* L_j := \delta _{k,j} I _{\ell ^2 (\F ^d )}, $$ so that the $d-$tuple
$$ L := (L_1 , \cdots , L _d ) : \ell ^2 (\F ^d ) \otimes \C ^d \rightarrow \ell ^2 (\F ^d ), $$ 
defines a \emph{row isometry}, \emph{i.e.} an isometry from several copies of a Hilbert space into itself.  In fact, Popescu's extension of the Wold decomposition shows that the left free shift, $L$, is the universal pure row isometry (with $d$ components) \cite{Pop-dil}. 

The left free shifts do not commute, and it may appear, at first sight, that one loses the connection to analytic function theory (since one cannot represent $L_k$ as $M_{z_k}$ on a space of holomorphic functions on some domain in $\C ^d$). Most remarkably, this is not the case -- instead $\ell ^2 (\F ^d )$ can be identified with a space of holomorphic functions in several non-commuting (NC) variables, $Z = (Z_1, \cdots  ,Z_d)$, and the $L_k \simeq M^L _{Z_k}$ become left multiplication by these independent NC variables. Certainly any element of $\ell ^2 (\F ^d )$ can be viewed as a formal power series in $d$ non-commuting variables $\mf{z} := (\mf{z} _1 , \cdots , \mf{z} _d )$:
$$ f := \sum _{\alpha \in \F ^d} \hat{f} _\alpha e _{\alpha} \mapsto \sum \hat{f} _\alpha \mf{z} ^{\alpha} =: f (\mf{z} ), $$ where if $\alpha = i_1 i_2 \cdots i_n$, $i _k \in \{ 1, ... ,d  \}$, we use the standard notation $\mf{z} ^\alpha = \mf{z} _{i_1} \mf{z} _{i_2} \cdots \mf{z} _{i_d}$. This is simply a change in notation, however, if $Z := (Z_1 , \cdots , Z_d ) : \cH \otimes \C ^d \rightarrow \cH$ is any strict row contraction on a Hilbert space, $\cH$, \emph{i.e.}
$$ Z Z^* = Z_1 Z_1 ^* + \cdots + Z_d Z_d ^* < I _\cH, $$ then  
the Popescu-von Neumann inequality for free polynomials,
$$ \| p (Z) \| \leq \| p (L ) \|; \quad \quad p \in \C \{ \mf{z} _1 , \cdots , \mf{z} _d \}, $$ 
implies that the power series for $f$ converges absolutely in operator norm when evaluated at $Z$ (and uniformly in operator norm on the ball of all strict row contractions of norm at most $r$, for any fixed $0<r<1$), $$ \| f(Z) \| ^2 \leq \frac{\| f \| ^2 _{\ell ^2 (\F ^d )} }{1 - \| Z \| ^2   }. $$ (The above inequality is a further consequence of the fact that $\| p (L ) \| = \| p \| _{\ell ^2 (\F ^d )}$ for any homogeneous free polynomial.) This shows that the power series, $f$, can be viewed as a function in the non-commutative open unit row-ball, $\B ^d _\N$, defined in Equation \ref{NCball}. Here, $\B ^d _n = \left( \C ^{n\times n} \otimes \C ^{1\times d} \right) _1$, can be viewed as the set of all strict row contractions with $d$ components on $\C ^n$. Any such square-summable power series (in particular any free polynomial) has the following three basic properties:
\bn
    \item $f : \B ^d _n \rightarrow \C ^{n \times n}$, $f$ is \emph{graded},
    \item If $Z \in \B ^d _n , W \in \B ^d _m$ then 
    $$ f \bpm Z & 0 \\ 0 & W \epm  = \bpm f(Z) & 0 \\ 0 & f (W) \epm, \quad \quad \mbox{$f$ \emph{respects direct sums},}$$
    \item If $W = S ^{-1} Z S$ is jointly similar to $Z$, \emph{i.e.} $W_k = S^{-1} Z_k S$, $1\leq k \leq d$, then 
    $$ f (W) = S^{-1} f(Z) S, \quad \quad \mbox{$f$ \emph{respects similarities}.} $$
\en
\vspace{0.1cm}
In modern NC Function Theory \cite{KVV,Ag-Mc,Taylor,Taylor2,Voic,Voic2}, the above three properties are taken as the axioms defining a \emph{non-commutative function}, and $\B ^d _\N \subseteq \C ^d _{\N}$, is an example of an NC set (any subset of $\C ^d _{\N}$ which is closed under direct sums is an NC set). The interpretation of $\ell ^2 (\F ^d )$ as non-commutative power series was first developed by Popescu in \cite{Pop-freeholo,Pop-freeholo2,Pop-freeharm}, and this has been one of the inspirations of non-commutative function theory. This general philosophy of identifying certain abstract Hilbert spaces and operator spaces with concrete spaces of NC functions has been a fruitful viewpoint in non-commutative analysis \cite{Voic,Voic2,Tak-dual,VinPop,SSS}.

Remarkably, any locally bounded NC function (on say a left and right admissable NC domain, see \cite{KVV}) is automatically holomorphic in the sense that it is both G\^{a}teaux and Fr\'{e}chet differentiable at any point in its domain and has a convergent Taylor-type power series expansion about any point \cite[Chapter 7]{KVV}. In particular, any $f \in \ell ^2 (\F ^d )$ defines a holomorphic NC function in $\B ^d _\N$. Much of classical complex analysis and several complex variables extends naturally to the setting of NC holomorphic functions with purely algebraic proofs including the Scwharz Lemma, Cauchy's estimates, Liouville's theorem (and much more) \cite{Pop-freeholo,Pop-freeholo2}, Hilbert's Nullstellensatz (this is in some sense `perfect' in the NC setting) and a maximum modulus principle \cite{SSS}, Oka-Weil theorem \cite{Ag-Mc}, and the concept of a complex analytic manifold \cite{AgMcY}.

It follows that we can identify $\ell ^2 (\F ^d)$ with the \emph{Non-commutative (NC)} or \emph{free Hardy space}:
$$ H^2 (\B ^d _\N ) := \left\{  f \in \scr{O} (\B ^d _\N )  \left| \ f(Z) = \sum _{\alpha \in \F ^d } \hat{f} _\alpha Z^\alpha, \ \sum |\hat{f} _\alpha | ^2 < \infty \right.  \right\}, $$ the Hilbert space of all (locally bounded hence holomorphic) NC functions in the NC unit ball $\B ^d _\N$ with square-summable MacLaurin series coefficients.  Under this identification the left free shifts, $L_k$, become left multiplication by the independent variables, $L_k = M^L _{Z_k}$. Note that $\ell ^2 (\F ^d )$ is also canonically isomorphic to the full Fock space over $\C ^d$:
$$ \hardy  :=  \C \oplus \C ^d \oplus \left( \C ^d \otimes \C ^d \right) \oplus \left( \C ^d \otimes \C ^d \otimes \C^d \right) \oplus \cdots = \bigoplus _{k=0} ^\infty \left( \C ^d \right) ^{k \cdot \otimes}, $$  the direct sum of all tensor powers of $\C ^d$. This isomorphism is implemented by the unitary map $e _\alpha \mapsto L^\alpha 1$, where $1$ denotes the vacuum vector (which spans the subspace $\C \subset \hardy$) of the Fock space.
Under this isomorphism the left free shifts become the left creation operators which act as tensoring on the left by the members of the standard orthornormal basis of $\C ^d$. In the sequel we identify $\ell ^2 (\F ^d ), \hardy$ and $H^2 (\B ^d _\N )$ and we will use these notations interchangeably.

\subsection{Fock space as a Non-commutative RKHS} \label{rkhs}

As in the single-variable setting, the free Hardy space $H^2 (\B ^d _\N )$ can be equivalently defined using (non-commutative) reproducing kernel theory \cite{BMV}. The theory of Non-commutative reproducing kernel Hilbert spaces (NC-RKHS) is a faithful analogue of the classical theory \cite{Aron-RKHS,Paulsen-rkhs}. Here, recall that a reproducing kernel Hilbert space is a Hilbert space of (complex-valued) functions, $\cH$, on a set $X$, so that point evaluation at any $x \in X$ is a bounded linear functional: for any $f \in \cH$ and $x \in X$,
$$ \ell _x (f ) := f(x); \quad \quad \ell _x \in \cH ^\dag. $$ 
The Riesz representation lemma then implies there is a corresponding \emph{kernel vector} $k_x \in \cH$ so that 
$$ \ip{k_x}{f}_\cH = \ell _x (f) = f(x), $$ and one typically writes $\cH =: \cH (k)$ where the \emph{reproducing kernel} $k: X \times X \rightarrow \C$ is defined as:
$$ k(x,y) := \ip{k_x}{k_y}_\cH; \quad \quad x,y \in X. $$
Any reproducing kernel is an example of a \emph{positive kernel function} on $X$: A function $k:X \times X \rightarrow \C$ is a positive kernel function if given any finite subset $\{ x_1, \cdots , x_N \} \subset X$, the corresponding Gram matrix of the $k_{x_j}$ is positive semi-definite:
$$ 0 \leq [ k (x_i , x_j ) ] _{1\leq i,j \leq N}. $$ 
Conversely, starting with any positive kernel function, $k$, on $X$, there is a unique RKHS of functions on $X$ with reproducing kernel equal to $k$. (One simply defines functions $k_y (x):= k(x,y)$ for $x,y \in X$, and then takes the Hilbert space completion with respect to the inner product defined by $\ip{k_x}{k_y} := k (x,y)$.)

The concept of a non-commutative reproducing kernel Hilbert space (NC-RKHS) is analogous: Let $\Om := \bigsqcup \Om _n$, 
$$ \Om _n = \Om \bigcap \C ^{n\times n} \otimes \C ^{1\times d},$$ be any NC subset of the NC universe $\C ^d _{\N}$. (Recall an NC set is any subset of $\C ^d _\N$ which is closed under direct sums.) A Hilbert space of non-commutative functions, $\cH$ on $\Om$ (taking values in, say, $\C _{\N} = \bigsqcup \C ^{n\times n}$) is an NC-RKHS if point evaluation at any point $Z \in \Om _n$ is a bounded linear map from $\cH$ into the Hilbert space
$\C ^{n\times n}$ (equipped with the normalized trace or Hilbert-Schmidt inner product). Denote this evaluation map by $\ell _Z: \cH \rightarrow \C ^{n \times n}$ for $Z \in \Om _n$, and let $K_Z : = \ell _Z ^* : \C ^{n\times n} \rightarrow \cH$ be its Hilbert space adjoint. For $y, v \in \C ^n$ we can then define
$$ K \{ Z , y , v \} := K_Z (yv^*) \in \cH. $$ Furthermore, given $Z \in \Om _n, y, v \in \C ^n$ and $W \in \Om  _m,  x, u \in \C ^m$ define the linear map
$$ K(Z,W) [ \cdot ] : \C ^{n\times m } \rightarrow \C ^{n\times m}, $$ by 
$$ \ipcn{y}{K(Z,W)[vu^*]x} := \ip{K \{Z , y, v \} }{ K \{ W , x, u \} } _{\cH}. $$
This defines a completely bounded linear map $K(Z,W) : \C ^{n\times m} \rightarrow \C ^{n\times m}$ so that $K(Z,Z) : \C ^{n\times n} \rightarrow \C ^{n\times n}$ is completely positive for any fixed $Z \in \Om _n$. The map $K(Z,W)$ is called the \emph{completely positive non-commutative reproducing kernel} (CPNC kernel) for the space $\cH$. The CPNC kernel is a sort of two-argument NC function, see \cite[Sections 2.3-2.4]{BMV} for details. As in the classical theory there is a bijection between CPNC kernel functions on a given NC set and NC-RKHS on that set \cite[Theorem 3.1]{BMV}, and if $K$ is a given CPNC kernel on an NC set, we will use the notation $\cH _{nc} (K)$ for the corresponding NC-RKHS of NC functions. In particular, the free Hardy space is the unique NC-RKHS corresponding to the \emph{NC Szeg\"{o} kernel}:
\be K(Z,W) [ \cdot ] := \sum _{\alpha \in \F ^d} Z^\alpha [ \cdot ]  W^{\alpha *}; \quad H^2 (\B ^d _\N ) = \cH _{nc} (K). \label{NCSzego} \ee 

All NC-RKHS in this paper will be NC-RKHS of free holomorphic functions in the NC unit ball $\B ^d _\N$ so that if $f \in \cH _{nc} (K)$, $f$ has a Taylor-MacLaurin series at $0 \in \B ^d _1$ with non-zero radius of convergence, \cite[Chapter 7]{KVV}: 
$$ f(Z) = \sum _{\alpha \in \F ^d } Z^\alpha \hat{f} _\alpha ; \quad \quad Z \in \B ^d _n, \ \hat{f} _\alpha \in \C, $$ 
and the linear coefficient evaluation fuctionals: 
$$  f \stackrel{\ell _\alpha}{\rightarrow} \hat{f} _\alpha; \quad \quad \alpha \in \F ^d, $$ are all bounded. We will let $K_\alpha $ denote the \emph{coefficient evaluation vector}:
$$ \ip{K_\alpha}{f}_{\cH _{nc} (K)} = \ell _\alpha (f) = \hat{f} _\alpha, \quad \quad \alpha \in \F ^d, $$ and we will typically write $\ell _\alpha =: K_\alpha ^*$. If $K$ is the NC Szeg\"{o} kernel of the free Hardy space, then
$$ K_\alpha (Z) = Z^\alpha, $$ \emph{i.e.} $K_\alpha$ can be identified with the free monomial $L^\alpha 1 \in \hardy$.

Recall that any reproducing kernel Hilbert space (RKHS) $\cH (k)$ on a set $X$ is naturally equipped with a \emph{multiplier algebra} $\mr{Mult} \, \cH (k) $, the algebra of all functions on $X$ which `multiply' $\cH (k)$ into itself: 
$$ h \in \cH (k), F \in \mr{Mult} \, \cH (k)  \ \Rightarrow \ F h \in \cH (k). $$ Any multiplier $F \in \mr{Mult} \, \cH (k) $ can be identified with a bounded linear multiplication operator $M_F \in \scr{L} (\cH (k) )$, and under this identification $\mr{Mult} \, \cH (k) $ is closed in the weak operator topology and unital. One can similarly define left and right multiplier algebras in the NC setting. Namely, if $\mc{H} _{nc} (K)$ is an NC-RKHS on an NC set $\Om$, then NC functions $F, G$ are left or right multipliers of $\cH _{nc} (K)$, respectively, if the NC functions
$$ (F \cdot h ) (Z) = F(Z) h (Z), \quad \mbox{or} \quad (h \cdot G) (Z) = h(Z) G(Z), $$ belong to $\cH _{nc} (K)$, for every $h \in \cH _{nc} (K)$. As in the classical theory, the adjoints of both left and right free multipliers have a natural action on point and coefficient evaluation vectors:
$$ (M^L _{F}) ^* K \{ Z , y , v \} = K \{ Z , F(Z)^* y , v \}, \quad \quad (M^R _G) ^* K \{Z , y, v \} = K \{ Z , y , G(Z) v \}, $$ and if $F(Z) = \sum Z^\alpha F_\alpha$, $G(Z) = \sum Z^\alpha G_\alpha$ then,
$$ (M^L _{F}) ^* K _\alpha = \sum _{\beta \ga = \alpha} K_\ga F_\beta ^*, \quad \quad (M^R _G) ^* K_\alpha = \sum _{\ga \beta = \alpha} K_\ga G_\beta ^*. $$ 
The left multiplier algebra of the free Hardy space provides a non-commutative generalization of $H^\infty (\D) = \mr{Mult} \, H^2 (\D )  $:
$$ H^\infty (\B ^d _\N ) :=  \left\{ f \in \scr{O} (\B ^d _\N )  \left| \ \sup _{Z \in \B ^d _\N } \| f (Z ) \| < \infty \right. \right\} = \mr{Mult} _L \, H^2 (\B ^d _\N ) . $$ As in the single variable setting, the left multiplier norm on $H^\infty (\B ^d _\N )$ (the norm of a left multiplier viewed as a left multiplication operator) coincides with the supremum norm in the NC unit ball \cite[Theorem 3.1]{SSS}, \cite[Theorem 3.1]{Pop-freeholo}. In keeping with the notation $\bH ^2 _d$ for the Fock space, we will often use the more compact notation $\mult = H^\infty (\B ^d _\N )$. This left multiplier algebra can also be identified with 
$$ L^\infty _d := \mr{Alg} \{ I, L_1, \cdots , L_d \} ^{-weak-*},$$ the (left) \emph{free analytic Toeplitz algebra}.  Let $$\A := \mr{Alg} \{ I , L_1, \cdots, L_d \} ^{-\| \cdot \| }, $$ be the left \emph{free disk algebra}. The free disk algebra can be viewed as the set of all uniformly bounded NC holomorphic functions in $\B ^d _\N$ which extend continuously to the boundary, $\partial \B ^d _\N$ (of all row contractions with unit norm). Recall that one can also define $R_k = M^R _{Z_k}$, the isometric \emph{right free shifts} on $H^2 (\B ^d _\N )$, and these are unitarily equivalent to the left free shifts via the self-adjoint transpose unitary on $\ell ^2 (\F ^d )$, $U_{\mr{t}}$,
$$ U_\mr{t} e_\alpha := e_{\alpha ^{\mr{t}}}, $$ where if $\alpha = i_1 \cdots i_n \in \F ^d$, then $ \alpha ^{\mr{t}} := i_n \cdots i_1, $ its transpose. We also define $R^\infty _d := U_{\mr{t}} L^\infty _d  U_{\mr{t}}$, the right free analytic Toeplitz algebra. Note that if $F(L) = M^L _{F(Z)} \in L^\infty _d$ has Taylor-MacLaurin series:
\be F(Z) = \sum _{\alpha \in \F ^d} F_\alpha Z^\alpha, \label{TMseries} \ee then its transpose-conjugate, $F^{\mr{t}} \in (\mult ) ^{\mr{t}}$ is a locally bounded NC function,
\be F^{\mr{t}} (Z) := \sum F_{\alpha ^{\mr{t}}} Z^\alpha, \label{transcon} \ee so that
$$ F(R) := U_{\mr{t}} F(L) U_{\mr{t}} = M^R _{F^{\mr{t}} } \in R^\infty _d. $$

As in \cite{DP-inv,Pop-entropy} a left (or right) free multiplier of the free Hardy space will be called \emph{inner} if the corresponding multiplication operator is an isometry, and \emph{outer} if the corresponding (left or right) multipication operator has dense range. 

\section{Non-commutative measures} \label{NCmeasuresect}

As described in the introduction, any finite, positive, and regular Borel measure, $\mu$, on $\partial \D$ is the Clark measure, $\mu = \mu _b$ corresponding to a contractive analytic function $b$ in $\D$. As before, we can identify $\mu$, via its moments, with a positive linear functional, $\hat{\mu}$, on the disk algebra operator system $\mc{A} (\D ) + \mc{A} (\D ) ^*$. (We write this in place of its norm-closure, which is simply the $C^*-$algebra $\mc{C} (\partial \D )$ of continuous functions on $\partial \D$ in this case.) 
If $b $ is a contractive analytic function, then $1-b$ is outer (cyclic for the shift), and it follows that 
$H_b (S) := (I + b(S) ) (I - b(S)) ^{-1}$ is a closed operator affiliated to $S$ with dense 
domain $\ran{I-b(S)}$. Moreover, it is not difficult to verify the formula:
\ba \hat{\mu} _b (S^k ) & = & \frac{1}{2} \left( \ip{H_b (S) ^*1}{S^k 1}_{H^2} + \ip{1}{H_b (S) ^* S^k 1} _{H^2} \right) \nn \\
& = & \delta _{k,0} \frac{1}{2} H_b (0) + \frac{1}{2}\ip{1}{H_b (S) ^* S^k 1} _{H^2}. \nn \ea
One can further re-write the Herglotz representation formula in terms of this Clark functional:
$$ H_b (z) = i \nbim H_b (0) + \hat{\mu} _b \left( (I + z S^*) (I-zS^* ) ^{-1} \right), $$ and the (conjugate) moments of $\mu$ are in fixed proportion to the MacLaurin series coefficients of $H_b$:
$$ H_b (z) = i \nbim H_b (0) + \hat{\mu} _b (I ) + 2 \sum _{k=1} ^\infty z^k \hat{\mu} _b (S^k ) ^*.$$

These constructions have exact analogues in the NC multi-variable setting. In place of $\mc {A} (\D ) = \mr{Alg} \{ I , S \} ^{-\| \cdot \|}$ we have the free disk algebra, $\A =\mr{Alg} \{ I, L_1 , \cdots , L_d \} ^{-\| \cdot \| }$. The results in this section can be found in \cite{JM-freeCE,JM-freeAC}.

\begin{defn} \label{NCmeasdef}
A (positive) \emph{non-commuative measure} is a positive linear functional on the \emph{free disk system}:
$$  \scr{A} _d := \left( \A + \A ^* \right) ^{-\| \cdot \|}. $$ The set of all positive NC measures will be denoted by $\posncm$. 
\end{defn}

\begin{defn}{ (\cite[Definition 3.1]{JM-freeCE})} \label{NCClarkdef}
Given any contractive free holomorphic $B \in \scr{L} _d := [ H^\infty (\B ^d _\N ) ] _1$, the \emph{Clark functional} or \emph{NC Clark measure} of $B$ is the positive linear functional, $\mu _B \in \posncm$, 
defined by:
\ba  \mu _B (L^\alpha ) & := & \frac{1}{2} \left( \ip{H_B (R) ^* 1 }{L^\alpha 1 } _{\hardy} + \ip{1}{H_B (R) ^* L^\alpha 1 } _{\hardy} \right) \nn \\
& = & \frac{1}{2} H_{B ; \emptyset} \cdot \delta _{\alpha , \emptyset} + \frac{1}{2} \ip{1}{H_B (R) ^* L^\alpha 1 } _{\hardy}. \nn \ea
The convex set, $\scr{L} _d := [\mult ]_1$, of all contractive NC  functions in $\B ^d _\N$ is called the \emph{left NC Schur class} or \emph{left free Schur class} \cite[Section 3]{Ball-Fock}. Similarly, $\scr{R} _d = \scr{L} _d ^{\mr{t}}$ will denote the \emph{right NC Schur class}, the set of all transpose-conjugates of elements of $\scr{L} _d$.
\end{defn}
It is not immediately obvious that the above definition of $\mu _B$ produces a positive NC measure or positive linear functional on the free disk system, $\scr{A} _d$. This is proven in \cite[Proposition 3.2]{JM-freeCE} and \cite[Proposition 4.5]{JM-freeAC}. Classically, the closed unit ball of $H^\infty (\D )$, \emph{i.e.} the closed convex set of all contractive analytic functions in the disk, is called the \emph{Schur class} \cite{Ball-Fock}. This motivates our terminology above.
\begin{remark} \label{leftvsright}
The left and right NC Schur classes are distinct. A simple example is given by the NC polynomial:
$$ B(Z) := \frac{1}{\sqrt{2}} Z_2 (I_n - Z_1); \quad Z \in \B ^d _n. $$ This is inner as a right multiplier, \emph{i.e.} $M^R _{B(Z)}$ is an isometry on $\hardy$, and hence has operator norm $1$. However, as a left multiplier, $M^L _{B(Z)}$ has norm $\sqrt{2} > 1$, see \cite[Example 3.4]{JM-ncld} for details.
\end{remark}
\begin{defn}
    A free holomorphic function, $H$ in $\B _\N ^d$ is a left free or NC Herglotz function if $\nbre H(Z)$ is positive semi-definite for
all $Z \in \B _\N ^d$. The set of all left free Herglotz functions is a positive cone which we denote by $\scr{L} _d ^+$. 
\end{defn}
We will also consider the right NC Herglotz class, $\scr{R} _d ^+$, the image of $\scr{L} _d ^+$ under the transpose involution, $\scr{R} _d ^+ = ( \scr{L} _d ^+ ) ^\mr{t}$. As in the classical setting, the fractional linear Cayley Transform implements a bijection between $\scr{L} _d$ and $\scr{L} _d ^+$: Given $B \in \scr{L} _d$,
$$ H_B := (I-B) ^{-1} (I+B) \in \scr{L} _d ^+, $$ and given $H \in \scr{L} _d ^+$,
$$ B_H := (H+I) ^{-1} (H-I) \in \scr{L} _d. $$ Similarly the Cayley Transform maps $\scr{R} _d$ bijectively onto $\scr{R} _d ^+$, and since the left and right NC Schur classes are distinct (see Remark \ref{leftvsright} above), so are the left and right NC Herglotz classes. Given $Z \in \B ^d _n$, let 
$$ ZL^* := \bsm Z _1 \otimes I_{\bH^2}, & \cdots, & Z_d \otimes I_{\bH ^2} \esm \bsm I_n \otimes L_1 ^* \\ \vdots \\ I_n \otimes L_d ^* \esm = Z_1 \otimes L_1 ^* + \cdots + Z_d \otimes L_d ^* \in \scr{L} (\C ^n \otimes \hardy ), $$ and set $I_{n\times \bH ^2} := I_n \otimes I_{\bH ^2 }$. Also note that $ZL^*$ is a strict contraction so that
$$ (I -ZL^* ) ^{-1} = \sum _{k=0} ^\infty (ZL^* ) ^k = \sum _{\alpha \in \F ^d }  Z^\alpha \otimes L^{*\alpha}, $$ is a convergent geometric series. The following result extends the classical bijection between Herglotz functions in the disk and positive measures on the circle to our NC multivariable setting.
\begin{thm}{ (\cite[Theorem 3.4]{JM-freeCE}, \cite[Section 5]{Pop-freeharm})} \label{freeHergthm}
The map $B \mapsto \mu _B$ is a bijection, modulo the imaginary part of $H_B (0)$, from $\scr{L} _d$ onto $\posncm$, 
and one has the \emph{NC Herglotz formula}: 
\be H_B (Z) = i \nbim H_B (0_n) + (\mr{id} _n \otimes \mu _B ) \left( (I_{n \times \bH ^2} + ZL^* ) (I_{n \times \bH ^2} -ZL^* ) ^{-1} \right); \quad Z \in \B ^d _n. \label{freeHerglotz} \ee
\end{thm}
In the above, $\mr{id} _n : \C ^{n\times n} \rightarrow \C ^{n\times n}$ denotes the identity map. Any NC measure $\mu \in \posncm$ is the NC Clark measure of some contractive NC holomorphic function $B \in \scr{L} _d$, $\mu = \mu _B$, and the moments of $\mu$ can be identified with the MacLaurin (Taylor-Taylor series at $0 \in \B ^d _1$) series coefficients of $H_B$: 
\be H_B (Z) = i \nbim H_{B} (0) I_n +  \mu (I)  I_n + 2 \sum _{\alpha \neq \emptyset } Z^\alpha \mu (L^{\alpha ^{\mr{t}} }) ^*, \label{HergTTseries} \ee
see \cite[Lemma 3.3]{JM-freeCE}. The left NC Herglotz-Riesz transform of $\mu \in \posncm$ is defined as
$$ H_\mu (Z) =  (\mr{id} _n \otimes \mu _B ) \left( (I_{n \times \bH ^2} + ZL^* ) (I_{n \times \bH ^2} -ZL^* ) ^{-1} \right); \quad Z \in \B ^d _n, $$ so that if $\mu = \mu _B$ for $B \in \scr{L} _d$ then $H_B = H_\mu + i \nbim H_B (0)$.

\subsection{Non-commutative Lebesgue measure} 
\label{NCLebesgue}
Classically, the Herglotz-Riesz transform, $H_m (z)$, of normalized Lebesgue measure, $m$ on $\partial \D$ is the constant function $H_m \equiv 1$:
\ba H_m (z) & := & \int _{\partial \D } \frac{1+z\zeta ^*}{1-z\zeta ^*} m (d\zeta ) \nn \\
& = & 2\sum _{k=0} ^\infty z^k \int _{\partial \D } (\zeta ^*) ^k m (d\zeta ) - \int _{\partial \D} m (d\zeta ) \nn \\
& = & 2 \sum _{k=0} ^\infty z^k \delta _{k,0} - 1 =1. \nn \ea
The corresponding contractive analytic function in $\D$ obtained as the inverse Cayley Transform of $H_m (z) \equiv 1$ is then identically $0$:
$$ b_m (z) := \frac{H_m (z) -1}{H_m (z) +1 } = 0. $$ 
It is then natural to expect that in the NC multi-variable theory, the role of normalized Lebesgue measure should be played by the unique positive NC measure corresponding to the constant free holomorphic functions: 
$$ B (Z) := 0 _n, \quad \ \mbox{or equivalently} \  \quad H_B (Z) := I_n; \quad \quad Z \in \B ^d _n. $$ 
Using the NC Herglotz representation formula (\ref{freeHerglotz}), it is easy to check that the unique NC measure (which we also denote by $m$), corresponding to the contractive NC function $B(Z) = 0_n $ is the vacuum state on the Fock space:
$$ m (L^\alpha ) := \ip{1}{L^\alpha 1} _{\bH ^2} = \delta _{\alpha, \emptyset}. $$ 

More evidence that this is to be expected, is that if $m$ is normalized Lebesgue measure on $\partial \D$, then the corresponding linear functional $\hat{m}$ restricted to $\mc{A} (\D )$ can be expressed as:
$$ \hat{m} (S^k ) = \ip{1}{S^k 1} _{H^2} = \delta _{k, 0}. $$

\begin{defn}
    The vacuum state $m \in \posncm$ will be called (normalized) \emph{NC Lebesgue measure}.
\end{defn}

\subsection{Left regular representations of the Cuntz-Toeplitz algebra}

If $\mu$ is any positive finite and regular Borel measure on $\partial \D$, it is natural to consider the $L^2$ space, $L^2 (\mu , \partial \D )$, as well as its `analytic part', 
$$ H^2 (\mu ) := \bigvee _{k\geq 0} \zeta ^k, $$ the closed linear span of the `analytic polynomials' in $L^2 (\mu , \partial \D )$. The operator $M _\zeta$, of multiplication by $\zeta $ is unitary on $L^2 (\mu)$, and $H^2 (\mu )$ is $M_\zeta -$invariant so that the restriction of $M_\zeta$ to $H^2 (\mu )$ is an isometry. 

When $d>1$, the appropriate analogues of $H^2 (\mu )$ and $M_{\zeta } | _{H^2 (\mu )}$ are obtained via a Gelfand-Naimark-Segal (GNS) construction:
If $\mu \in \posncm$, the fact that the free disk algebra has the semi-Dirichlet property \cite{DK-dilation}: $$ \A ^* \A \subseteq (\A + \A ^*) ^{-\| \cdot \| }, $$ ensures that the GNS pre-inner product:
$$ \ip{a_1}{a_2} _\mu := \mu (a_1 ^* a_2); \quad \quad a_1, a_2 \in \A $$ is well-defined. The GNS space $\hardy (\mu _B )$ is then the Hilbert space completion of $\A$ modulo zero length vectors with respect to this pre-inner product. The equivalance class of $a \in \A$ will be denoted by $a + N_\mu$, where $N_\mu \subseteq \A$ is the left ideal of all elements of zero length.  Moreover, the left regular representation: $\pi _\mu : \A  \rightarrow \scr{L} (\hardy (\mu  ) )$, 
$$ \pi _\mu (a_1) (a_2 + N_\mu) := a_1 a_2 + N _\mu; \quad \quad a_1, a_2 \in \A, $$ is completely isometric and extends to a $*-$representation of the Cuntz-Toeplitz algebra $\mc{E} _d = C^* (I, L)$ on $\scr{L} (\hardy (\mu ) )$ \cite{Cuntz}. In particular, 
$$ \Pi _\mu = \pi _\mu (L) := (\pi _\mu (L_1) , \cdots , \pi _\mu (L_d) ) : \hardy (\mu ) \otimes \C ^d \rightarrow \hardy (\mu ), $$ is a (row) isometry, and we write $\Pi _{\mu ; k} := \pi _\mu (L_k).$ Again, if $d=1$ then 
$$ \bH ^2 _1 (\hat{\mu } ) \simeq H^2 (\mu ), \quad \mbox{and} \quad \Pi _{\hat{\mu}} \simeq M_\zeta | _{H^2 (\mu)}, $$ where $\hat{\mu}$ is, as before, the positive linear functional corresponding to the positive measure, $\mu$ on the circle $\partial \D$.

\subsection{The (right) NC Herglotz class} 
\label{NCHergspace}

For our purposes, it will be convenient to consider the right free Herglotz class $\scr{R} _d ^+ = (\scr{L} _d ^+ ) ^{\mr{t}}$, the image of the left free Herglotz class under the involutive transpose map. Namely if $H \in \scr{L} _d ^+$ has Taylor-MacLaurin series:
$$ H(Z) = \sum _{\alpha \in \F ^d} Z^\alpha H_\alpha =  \mu (I) I_n + 2 \sum _{\alpha \neq \emptyset } Z^\alpha \mu (L^{\alpha ^{\mr{t}} }) ^*, $$  then $H^{\mr{t}} \in \scr{R} _d ^+$ has Taylor-MacLaurin series:
$$ H^{\mr{t}} (Z) = \sum _\alpha Z^\alpha H_{\alpha ^{\mr{t}}} = \mu (I) I_n + 2 \sum_{\alpha \neq \emptyset} Z^\alpha \mu (L^\alpha ) ^*, $$ see \cite[Section 3]{JM-freeCE}. As in \cite{JM-freeCE,JM-freeSmirnov,JM-F2Smirnov}, we can identify $\scr{R} _d ^+$ as closed (potentially unbounded) right multiplication operators densely-defined in the full Fock or NC Hardy space: if $H  \in \scr{R} _d ^+$, then $M^R _{H (Z)} = H^\mr{t} (R)$, where $H ^\mr{t} (L) = M^L _{H ^\mr{t} (Z)}$ and $H^\mr{t} \in \scr{L} _d ^+$. Given $H  \in \scr{R} _d ^+$, one can construct the (right) \emph{free Herglotz space}, $\scr{H} ^+ (H ) := \cH _{nc} (K ^H )$, the unique NC-RKHS corresponding to the (right) \emph{free Herglotz kernel}: 
$$ K^H (Z,W) := \frac{1}{2} K(Z,W) \left[H (Z) (\cdot)  +  (\cdot) H (W) ^* \right]; \quad \quad Z,W \in \B ^d _\N, $$ see \cite[Section 4]{JM-freeCE},\cite{JM-freeAC}. Here $K(Z,W)$ denotes the CPNC Szeg\"{o} kernel in the NC unit ball $\B ^d _\N$, see Equation \ref{NCSzego}, and $K^H$ is also a CPNC kernel in $\B ^d _\N$. 
\begin{lemma}
A locally bounded NC function, $H$, in $\B ^d _\N$ belongs to the right free Herglotz class if and only if $K^H$ is a CPNC kernel.
\end{lemma}
There is an analogous kernel characterization of the left NC Herglotz class, see \cite[Section 3]{JM-freeCE}.
\begin{proof}
By definition $H$ is right Herglotz if and only if $H^\mr{t}$ is left Herglotz, \emph{i.e.} if and only if $\nbre H^\mr{t} (Z) \geq 0$ for all $Z \in \B ^d _\N$. By \cite[Chapter III.8]{NF}, the Cayley Transform implements a bijection between closed, accretive operators (operators with numerical range in the right half-plane) which are densely-defined in a Hilbert space and contractions which do not have $1$ as an eigenvalue. In particular, $H^\mr{t}$ is a left NC Herglotz function if and only if $B ^\mr{t} := (H^\mr{t} -1) (H^\mr{t} +1 ) ^{-1} \in \scr{L} _d$ is a contractive NC function in $\B ^d _\N$. (By the NC maximum modulus principle, \cite[Lemma 6.11]{SSS}, any non-constant $B^\mr{t} \in \scr{L} _d$ must take strictly contractive values in $\B ^d _\N$.) Setting $B = (B^\mr{t} ) ^\mr{t}$, we conclude that $H \in \scr{R} _d ^+$ belongs to the right NC Herglotz class if and only if $M^R _B = B^\mr{t} (R) = U_\mr{t} B ^\mr{t} (L) U_\mr{t}$ is a contractive right multiplier so that $B \in \scr{R} _d$ belongs to the right Schur class.

It is well-known that a multiplier, $b$, of a reproducing kernel Hilbert space, $\mc{H} (k)$, is contractive, if and only if the \emph{de Branges-Rovnyak kernel}, $k^b (z,w) := k(z,w) - b(z) k(z,w) b(w) ^*$ is a positive kernel function. It is straightforward to verify this in the NC setting: A right multiplier, $B \in (\bH ^\infty _d ) ^\mr{t}$, of $\bH ^2 _d$ is contractive if and only if $I_{\bH ^2} - M^R _B (M^R _B) ^* \geq 0$, so that $B \in \scr{R} _d$ if and only if for any $(Z,y ,v ) \in \B ^d _n \times \C ^n \times \C ^n$ and any $n \in \N$,
\ba
0 & \leq &  \ip{K\{Z, y ,v \}}{(I - M^R _B (M^R _B) ^*) K\{ Z , y ,v \}}_{\bH ^2} \nn \\
& = & \ip{K\{Z, y, v \}}{K\{ Z , y , v \}}_{\bH ^2} - \ip{K\{Z , y , B(Z) v \}}{K\{ Z , y , B(Z) v \} }_{\bH ^2} \\
& = & y^* K (Z,Z) [vv^*] y - y^* K(Z,Z) [B(Z) vv ^* B(Z) ^*] y. \nn \ea 
Clearly if $B$ is a contractive right multiplier the above expression is positive semi-definite. Conversely, linear combinations of NC Szeg\"o kernels are NC Szeg\"o kernels: $K\{ Z , y ,v \} +c K\{ W , x , u \} = K\{ Z \oplus W , y \oplus cx , v \oplus u \}$, and the linear span of the NC Szeg\"o kernels is dense. Hence, the above expression is positive semi-definite for all $(Z,y,v)$ if and only if $B \in \scr{R} _d$.

It follows that $B$ is a right NC Schur function if and only if the linear map $K^B (Z,Z) [\cdot ] : \C ^{n\times n} \rightarrow \C ^{n\times n}$ defined by
$$ K^B (Z,Z) [\cdot ] := K (Z,Z) [ \cdot ] - K(Z,Z)[ B(Z) (\cdot ) B(Z) ^*], $$ is positive semi-definite for any $Z \in \B ^d _n$ and $n \in \N$. By \cite[Subsection 2.4]{BMV}, this is equivalent to $K^B (Z,Z) [\cdot ]$ being completely positive for any $Z \in \B ^d _\N$, and hence to $K ^B (Z,W)$ being a CPNC kernel. It is clear that $B(Z) = (H(Z) - I) (H(Z) +I ) ^{-1}$ is the Cayley Transform of $H$, and a bit of algebra verifies that 
$$ K^B (Z,W) [\cdot ] = K^H (Z,W) [ (I - B(Z) ) (\cdot ) (I - B(W) ^*)], $$ or equivalently, $K^H (Z,W) [\cdot ] = K^B (Z,W) [ (I-B(Z) ) ^{-1} (\cdot) (I- B(W) ^* ) ^{-1} ]$. It follows easily from these equations that the NC de Branges-Rovnyak kernel $K^B$ is a CPNC kernel if and only if $K^H$ is a CPNC kernel. 
\end{proof}
Given $H \in \scr{R} _d ^+$, there is then a corresponding NC-RKHS, $\cH _{nc} (K^H)$, of NC holomorphic (\emph{i.e.} locally bounded) functions in $\B ^d _\N$. If $\mu \in \posncm$ is the unique NC measure corresponding to the right NC Herglotz function $H \in \scr{R} _d ^+$ by Theorem \ref{freeHergthm}, we will usually write $K ^H = K ^\mu$, and we will use the notation $\scr{H} ^+ (H_\mu ) := \cH _{nc} (K ^\mu )$ for the right free Herglotz space of $H_\mu$. Here, we will also write $H = H_\mu$ (or sometimes $\mu = \mu _H$). As described in \cite{JM-freeCE,JM-freeAC}, if $H = H _\mu$, there is a natural onto isometry, the (right) \emph{free Cauchy transform}, $\mc{C} _\mu : \hardy (\mu ) \rightarrow \scr{H} ^+ (H_\mu )$: For any free polynomial $p \in \fp \subseteq  \hardy (\mu )$,
\ba (\mc{C} _\mu p ) (Z) & = & (\mr{id} _n \otimes \mu \circ \mrt  ) \left( (I_{n\times \bH ^2} -Z \circ L^*) ^{-1} \, I_n \otimes p(L) \right) \nn \\
& := & \sum _{\alpha \in \F ^d } Z^\alpha \mu \left( L^{\alpha *} p(L)  \right) \nn \\
& = &  \sum _{\alpha \in \F ^d } Z^\alpha \ip{L^\alpha + N_\mu }{p(L) + N_\mu } _\mu.   \nn \ea
In the above, as before, for any $Z \in \B ^d _\N$, $ZL^* = Z_1 \otimes L_1 ^* + ... + Z_d \otimes L_d ^*$ is a strict contraction. The final formula above extends to arbitrary $x \in \hardy (\mu)$. (In the first line of the formula above, the $\mr{t}$ symbol means that one needs to take the transpose of all words in $L^*$ appearing in the geometric sum of $(I _{n \times \bH ^2} - ZL^* ) ^{-1}$ to obtain the second line.)  \\

\subsection{NC Cauchy Transform of GNS row isometry}

The image of the GNS row isometry $\Pi _\mu$ under the free Cauchy transform is a row isometry on the free Herglotz space:
\be V _\mu = \mc{C} _\mu \Pi _\mu \mc{C}  _\mu  ^* := \mc{C} _\mu \left( \Pi _{\mu ; 1} , \cdots , \Pi _{\mu ; d } \right) \mc{C} _\mu ^* \otimes I_d : \scr{H} ^+ (H _\mu ) \otimes \C ^d \rightarrow \scr{H} ^+ (H _\mu ), \ee where  $\Pi _{\mu ; k } = \pi _\mu (L_k)$. The range of the row isometry $V _\mu $ is:
\be \nbran V_\mu =\bigvee _{\substack{(Z,y,v) \in \B ^d _n \times \C ^n \times \C ^n, \\  \ n \in \N }} \left( K^{\mu} \{ Z , y , v \} - K^{\mu} \{ 0 _n , y, v \} \right) = \bigvee _{\alpha \neq \emptyset} K^{\mu } _\alpha, \ee and
for any $Z \in \B ^d _n, \ y, v \in \C^n$, 
\be 
V_\mu   ^* \left( K^{\mu} \{ Z , y , v \} - K^{\mu} \{ 0 _n , y, v \} \right) = K ^{\mu} \{ Z , Z ^* y ,  v \} := \bsm K ^{\mu} \{ Z , Z_1 ^* y ,  v \}  \\ \vdots \\ K ^{\mu} \{ Z , Z_d ^* y ,  v \} \esm \in \scr{H} ^+ (H _\mu ) \otimes \C ^d. \label{Vmuaction} \ee
The linear span of all such vectors is dense in $\scr{H} ^+  (H_\mu) \otimes \C ^d$ since $V_\mu ^*$ is a co-isometry.

The image of $\nbran V_\mu$ under $\mc{C} _\mu  ^*$ is $\hardy (\mu ) _0 = \bigvee _{\alpha \neq \emptyset} L^\alpha + N_\mu$, the closed linear span of the non-constant free monomials in $\hardy (\mu )$. If $F \in \scr{H} ^+ (H_\mu)$ is orthogonal to $\nbran V_\mu$, then for any $Z \in \B ^d _n$, 
$$ F(Z) = I_n F(0)  $$ \emph{i.e.} $F \equiv F(0) \in \C$ is constant-valued. See \cite[Section 4.4]{JM-freeCE} for details. 

\begin{remark}
Recall that if $\mu = m$ is normalized NC Lebesgue measure (the vacuum state), then $H_\mu (Z) = I_n $ for any $Z \in \B ^d _n$ so that the NC Herglotz kernel, $K^{m} = K$ reduces to the NC Szeg\"{o} kernel and $\scr{H} ^+ (H_m) = H^2 (\B ^d _\N ) $ is simply the free Hardy space. In this case $V_m \simeq M^L _{Z} \simeq L$ is the left free shift. 
\end{remark}

\section{Lebesgue decomposition of NC Toeplitz forms} \label{sec:simon}

Any positive NC measure $\mu \in \posncm$ can be identified with a positive semi-definite quadratic form, $q _\mu$, with dense domain, $\nbdom q_\mu = \A$, in the Fock space. In this section we define the absolutely continuous and singular parts of any $\mu \in \posncm$ (with respect to normalized NC Lebesgue measure, $m$) by applying B. Simon's Lebesgue decomposition theory for forms to $q_\mu$ \cite[Section 2]{Simon1}. Standard references for the theory of potentially unbounded quadratic forms in Hilbert space are \cite{Kato} and \cite[Section VIII.6]{RnS1}.

\subsection{Closable Toeplitz forms} \label{Closedforms}
All inner products and sesquilinear forms in this paper are conjugate linear in their first argument. A seqsuilinear (or quadratic) form, $q : \nbdom q \times \nbdom q \rightarrow \C$, where the domain of $q$, $\nbdom q \subseteq \mc{H}$, is dense in a Hilbert space, $\mc{H}$, is positive semi-definite if $q (h , h ) \geq 0$ for all $h \in \nbdom q$. Given such a positive semi-definite quadratic form, $q$, define $\cH (q+1)$ as the Hilbert space completion of $\nbdom q$ with respect to the inner product:
$$ \ip{x}{y} _{q+1} := q(x,y) + \ip{x}{y} _\cH. $$ The form $q$ is \emph{closed} if $\nbdom q $ is complete in the norm of $\| \cdot \| _{q+1}$, \emph{i.e.} if $\nbdom q = \cH (q+1)$.

A positive semi-definite quadratic form $q$, with dense domain in a Hilbert Space, $\cH$, is closed if and only if there is a unique closed, positive semi-definite operator $A$, with dense domain in $\cH$ so that $\nbdom q = \mr{Dom}\, \sqrt{A}$ and
$$ q (h ,g ) = q_A (h, g) := \ip{ \sqrt{A} h }{\sqrt{A} g} _\cH; \quad \quad g,h \in \nbdom q, $$
\cite[Chapter VI, Theorem 2.1, Theorem 2.23]{Kato}.  This can be viewed as an extension of the Riesz representation lemma for bounded positive semi-definite quadratic forms.
A positive quadratic form, $q$, is \emph{closable} if it has a closed extension.  Equivalently, $q$ is closable if and only if the following condition holds: If $x_n \in \nbdom q$ converges to $0$, $x_n \rightarrow 0$, and $x_n$ is also Cauchy with respect to the pre-inner product defined by $q$,
$$ q(x_n-x_m , x_n - x_m ) \rightarrow 0, $$ then this sequence also converges to $0$ with respect to $q$, \emph{i.e.} $ q(x_n , x_n ) \rightarrow 0$. If $q$ is closable, then it has a minimal closed extension, $\ov{q}$, with $\nbdom \ov{q} \subseteq \cH$ equal to the set of all $h \in \cH$ so that there is a sequence $h_n \in \nbdom q$, such that $h_n \rightarrow h$ and $(h_n)$ is Cauchy in the norm of $\cH (q+1)$.  A linear subset $\mc{D} \subseteq \nbdom q$ is called a \emph{form core} for a closed form $q$ if $\mc{D}$ is a dense linear subspace in $\cH (q +1 )$. It follows that if $q$ is closable with closure (minimal closed extension) $\ov{q}$, then $\nbdom q$ is a form core for $\ov{q}$ \cite[Chapter VI, Theorem 1.21]{Kato}. If $q = q_A$ is a closed, positive semi-definite quadratic form, then $\mc{D}$ is a form core for $q$ if and only if $\mc{D}$ is a core for $\sqrt{A}$. In particular, $\nbdom A$ is a form core for $q$. Here, recall that a linear subspace $\mc{D} \subseteq \nbdom A$ is called a \emph{core} for a closed operator $A$, if $\mc{D} \oplus A \mc{D} $ is dense in the (closed) graph of $A$. 

In \cite[Section 2]{Simon1}, B. Simon proved that any densely-defined positive semi-definite quadratic form, $q$, acting in a Hilbert space $\cH$, has a unique Lebesgue decomposition:
$$ q = q_{ac} + q_s; \quad \quad q_{ac}, q_s \geq 0$$ 
where $q_{ac}$ is the maximal closable form bounded above by $q$, $q_{ac} \leq q$ and $q_s = q - q_{ac} \leq q$. Here, a partial order on positive semi-definite quadratic forms with dense domains in a Hilbert space $\cH$ is defined by $q_1 \leq q_2$ if $\nbdom q_2 \subseteq \nbdom q_1$ and \be q_1 (h, h) \leq q_2 (h,h) \quad \forall \, h \in \nbdom q_2. \label{formpo} \ee 

\begin{thm}{ (\cite[Section 2]{Simon1}, \cite[Theorem S.15]{RnS1})} \label{maxclosform}
If $q$ is a positive semi-definite quadratic form, densely-defined in a Hilbert space $\cH$, then there is a maximal closable positive semi-definite form $q_{ac}$ bounded above by $q$, $q_{ac} \leq q$, and $\nbdom q$ is a form-core for $\ov{q_{ac}}$. If $E : \cH (q+1) \hookrightarrow \cH $ is the canonical embedding, and $Q_s$ is the orthogonal projection onto $\nbker E$, $q_{ac}$ is given by the formula:
$$ q_{ac} (h_1 , h_2 )  = \ip{ h_1}{Q_{ac} h_2} _{\cH (q+1)} - \ip{h_1}{h_2} _\cH; \quad \quad Q_{ac} := I - Q_s. $$
\end{thm}
\begin{remark}
In the above, maximality refers to the partial order defined in Equation (\ref{formpo}). This theorem yields the unique Lebesque decomposition
$$ q = q_{ac} +q_s; \quad \quad 0 \leq q_s = q-q_{ac} \leq q, $$ where $q_{ac}$ is \emph{absolutely continuous}, \emph{i.e.} closable, and $q_s$ is \emph{singular} in the sense that the only closable positive semi-definite quadratic form bounded above by $q_s$ is the identically zero form.  
\end{remark}
The proof of the above Lebesgue decomposition theorem for quadratic forms has similarity to von Neumann's approach to Lebesgue decomposition theory \cite[Lemma 3.2.3]{vN3}, and in fact recovers von Neumann's classical proof of the Lebesgue decomposition of a finite positive and regular Borel measure (on say $\partial \D$), $\mu$, with respect to another, $\la$, if one takes $\cH = L^2 (\la, \partial \D )$, and $q_\mu$ to be the quadratic form with the continuous functions $\scr{C} (\partial \D ) \subset L^2 (\la )$ as a dense form domain, 
$$ q_\mu (f,g) := \int _{\partial \D } \ov{f (\zeta )} g (\zeta ) \mu (d\zeta ), $$\cite{Simon1}. We are primarily interested in positive semi-definite quadratic forms arising from positive NC measures $\mu \in \posncm$:
\begin{defn} \label{LToepdef}
A positive semi-definite sesquilinear form, $q$, with dense $L-$invariant domain $\nbdom q \subseteq \hardy$ is called \emph{left} or \emph{$L-$Toeplitz} if:
\be q (L_j g, L_k h ) = \delta _{k,j} \, q (g , h); \quad \quad g,h \in \mr{Dom} \, q. \label{LToepform} \ee 

A closed, positive semi-definite and densely-defined operator $T : \nbdom T \subseteq \hardy \rightarrow \hardy$ will be called \emph{left} or \emph{$L-$Toeplitz} if:
\bn 
\item $\mr{Dom} \, \sqrt{T}$ is $L-$invariant, and
\item the associated closed quadratic form 
$$ q _{T} ( g , h ) := \ip{\sqrt{T} g }{\sqrt{T} h } _{\hardy}; \quad \quad g, h \in \nbdom \sqrt{T}, $$ is $L-$Toeplitz.
\en
\end{defn}
In particular, if $\mu \in \posncm$ is any positive NC measure, then
\be q_\mu (a_1, a_2 ) := \mu (a_1 ^* a_2); \quad \quad a_1, a_2 \in \A, \label{qmu} \ee is a positive semi-definite $L-$Toeplitz form with dense form domain $\nbdom q_\mu  = \A \subset \hardy$. Note that the left Toeplitz condition, Equation \ref{LToepform}, is equivalent to:
$$ q(L^\alpha g , L^\beta  h) = \left\{ \begin{array}{ccc} q(L^{\ga} g, h ) & \quad  & \mbox{if} \ \alpha = \beta \ga \ \mbox{for some} \ \ga \in \F ^d, \\ q (g , L^{\ga } h) & & \mbox{if} \ \beta = \alpha \ga, \\ 
    0 & &  \mbox{else.} \end{array} \right. $$
\begin{remark}
One could further define unbounded $L-$Toeplitz forms and operators which are not positive semi-definite, but we will have no need for this concept. Most $L-$Toeplitz forms, $q$, and operators, $T$, we consider will be such that the free polynomials, $\fp = \C \{ \mf{z} _1 , \cdots , \mf{z} _d \}$, and the free disk algebra, $\A$, are $L-$invariant cores for both $q$ and $\sqrt{T}$.
\end{remark}
\begin{remark} \label{Toedef} 
By the Riesz representation lemma, a bounded sesquilinear form on $\hardy$ is $L-$Toeplitz if and only if it is the quadratic form of a bounded positive semi-definite $L-$Toeplitz operator. Moreover, a bounded positive semi-definite operator, $T\in \scr{L} (\hardy)$, is $L-$Toeplitz if and only if
$$ L_k ^* T L_j = \delta _{k,j} T. $$ Such operators are called multi-Toeplitz in \cite{Pop-entropy}. Here, recall that a bounded operator $T$ on the Hardy space $H^2 (\D )$ is called \emph{Toeplitz} if $T = T_f = P_{H^2} M_f | _{H^2}$ for some $f \in L^\infty (\partial \D )$. A result of Brown and Halmos identifies the bounded Toeplitz operators as the set of all bounded operators $T \in \scr{L} (H^2 )$ with the \emph{Toeplitz property}:
$$ S^* T S = T, $$ where $S = M_z$ is the shift on $H^2$ \cite[Theorem 6]{BrownHalmos}. 
\end{remark}
\begin{lemma}\label{fpcore}
If $T$ is a closed positive semi-definite $L-$Toeplitz operator so that $\A$ is a core for $\sqrt{T}$, then $\mr{Dom} \, \sqrt{T}$ is $L-$invariant and $\fp$ is a core for $\sqrt{T}$.
\end{lemma}
\begin{proof}
By assumption, $\mc{A} _d$ is a core for $\sqrt{T}$. To see that $\mr{Dom} \sqrt{T}$ is $L-$invariant, given any $x \in \mr{Dom} \sqrt{T}$, choose $a_n \in \mc{A} _d$ so that $a_n \rightarrow x$ in $\hardy$, and $\sqrt{T} a_n \rightarrow \sqrt{T} x$. Then for any $1 \leq k \leq d$, $L_k a_n \in \A \subseteq \mr{Dom} \sqrt{T}$, $L_k a_n \rightarrow L_k x$, and 
\begin{align*} \| \sqrt{T} (L_k a_n - L_k a_m ) \| ^2 _{\bH ^2} & = q_T \left( L_k (a_n -a_m) , L_k (a_n -a_m) \right) \\
& = q_T ( a_n -a_m , a_n -a_m ) \quad \quad \mbox{($q_T$ is an $L-$Toeplitz form)} \\ 
& = \| \sqrt{T} (a_n - a_m ) \| _{\bH ^2} ^2, \end{align*} so that $\sqrt{T} L_k a_n$ is also Cauchy and converges to some $y \in \hardy$. Since $\sqrt{T}$ is closed, it follows that $\sqrt{T} L_k x = y$, and $\mr{Dom} \sqrt{T}$ is $L-$invariant. 

To show that $\fp$ is a core for $\sqrt{T}$, it suffices to show that the set of all $p \oplus \sqrt{T} p $ for $p \in \fp$ is dense in the set of all $a \oplus \sqrt{T} a$ for $a \in \mc{A} _d$, since this latter set is dense in the graph of $\sqrt{T}$.
Since $\mc{A} _d = \mr{Alg} \{ I , L_1 , \cdots , L_d \} ^{-\| \cdot \|}$, given any $a(L) \in \mc{A} _d$, there exists a sequence of free polynomials $p_n (L)$ so that $p_n (L) \rightarrow a (L)$ in operator norm. In particular, $p_n := p_n (L) 1 \rightarrow a := a(L) 1$ in $\hardy$. For each $1 \leq k \leq d$, define the linear map $\Pi _k$ on $\mr{Ran} \sqrt{T}$ by 
$$ \Pi _k \sqrt{T} x = \sqrt{T} L_k x; \quad \quad x \in \mr{Dom} \sqrt{T}. $$ (This is well-defined since $\mr{Dom} \sqrt{T}$ is $L-$invariant.) Since $T$ is $L-$Toeplitz, it is easy to check that the $\Pi _k$ define row isometries with pairwise orthogonal ranges: Given $x,y \in \mr{Dom} \sqrt{T}$,
\ba \ip{\Pi _k \sqrt{T}x}{ \Pi _j \sqrt{T}y} & = & \ip{\sqrt{T} L_k x}{\sqrt{T} L_j y} \nn \\
& = & \delta _{k,j} \, \ip{\sqrt{T}x}{\sqrt{T}y}. \label{Trowisom} \ea 
Hence, $\Pi := \left( \Pi _1 , \cdots , \Pi _d \right)$ extends by continuity to a row isometry on $\cH _T := \ov{\mr{Ran} \sqrt{T}}$. It follows that,
\ba  \| \sqrt{T} (p_n -p_m ) \|  _{\bH ^2} & = & \| \sqrt{T} (p_n (L) - p_m (L) ) 1 \| _{\bH ^2} \nn \\ 
& = & \|  \left( p_n (\Pi) - p_m (\Pi) \right) \sqrt{T} 1 \|   \nn \\ 
& \leq & \| p_n (\Pi ) - p_m (\Pi ) \| \| \sqrt{T} 1 \| \nn \\
& \leq & \| p_n (L) - p_m (L) \| \| \sqrt{T} 1 \|. \nn \ea In the above we have used that $\pi (L) = \Pi$ extends to a unital $*-$representation of the Cuntz-Toeplitz $C^*-$algebra, $\mc{E} _d := C^* \{ I, L_1 , \cdots , L_d \}$ and hence is completely contractive. (Alternatively, this also follows from a trivial application of Popescu's NC von Neumann inequality \cite{Pop-vN}.) Since the sequence $p_n (L)$ is Cauchy in operator norm, it follows that $(\sqrt{T} p_n )$ is Cauchy in $\hardy$ and converges to some $y \in \hardy$. As before, since $\sqrt{T}$ is closed, it follows that $\sqrt{T} a =y$. This proves that $\fp$ is a core for $\sqrt{T}$. 
\end{proof}

\subsection{Absolutely continuous and singular $L-$Toeplitz forms}

Let $\mu \in \posncm$ be any positive NC measure and let $q_\mu$ be the $L-$Toeplitz form with dense form domain $\mc{A} _d \subset \bH ^2 _d$ defined by $\mu$, see Equation \ref{qmu}. Then $q_\mu$ has a Lebesgue decomposition into absolutely continuous (closable) and singular parts given by Theorem \ref{maxclosform}. We now define an NC measure, $\mu$, to be absolutely continuous, or singular, if the $L-$Toeplitz form $q_\mu$ is absolutely continuous or singular, respectively.

\begin{defn} \label{NCncld}
A positive NC measure $\mu \in \posncm$ is:
\bn 
\item \emph{absolutely continuous} (AC) (with respect to NC Lebesgue measure, $m$) if $q_\mu$ is an absolutely continuous (\emph{i.e.} closable) quadratic form.

\item \emph{singular} (again with respect to $m$), if $q_\mu$ is singular, \emph{i.e.} the maximal absolutely continuous part of the $L-$Toeplitz form $q_\mu$ vanishes identically.
\en
\end{defn}

In particular, if $q_\mu = q_{ac} + q_s$ is the Lebesgue decomposition of the $L-$Toeplitz form $q_\mu$, then the explicit formula of Theorem \ref{maxclosform} shows that the absolutely continuous form $q _{ac} : \A \times \A \rightarrow \C$ can be expressed as:
\ba q _{ac} (a_1, a_2) & = & \ip{ a_1 + N_{\mu +m}}{Q_{ac} (a_2 + N_{\mu +m})}_{\mu +m} - \ip{a_1}{a_2}_{\bH ^2 _d} \nn \\
& = & \ip{a_1 + N_{\mu +m}}{(Q_{ac} - E^* E ) (a_2 + N_{\mu +m})}_{\mu +m}. \label{qacformula} \ea Here, 
$Q_s = I - Q_{ac}$ is the orthogonal projection onto the kernel of the contractive embedding $E : \hardy (\mu +m ) \hookrightarrow \hardy = \hardy (m)$ defined by $E (p(L) + N_{\mu +m} ) := p(L) 1 = p \in \hardy$. (Note that $E^* E \leq P_{\ov{\nbran E^*}} = Q_{ac}$.) In general, given an arbitrary positive NC measure $\mu$, it is not obvious whether or not there exist positive NC measures $\mu _{ac} ,\mu _s \in \posncm$ so that $q_{ac} = q_{\mu _{ac}}$, $q_s = q_{\mu _s}$ and $\mu = \mu _{ac} + \mu _s$. This will be established in the next section using NC Cauchy transform techniques, and this will yield a unique Lebesgue decomposition of any $\mu \in \posncm$ into the sum of absolutely continuous and singular (positive) NC measures, see Theorem \ref{NCLD}.

\begin{remark}
When $d=1$, and $\hat{\mu}$ is the linear functional on $\scr{C} (\partial \D)$ corresponding to a finite, positive and regular Borel measure, $\mu$, it is also not obvious that $q_{ac}$ and $q_s$, where $q_\mu = q_{ac} + q_s$ is the Lebesgue form decomposition, correspond to the linear functionals of the absolutely continuous and singular parts of the measure $\mu$ (with respect to normalized Lebesgue measure, $m$, on $\partial \D$), as constructed in classical Measure Theory. The fact that one recovers the classical Lebesgue decomposition in this way follows from the results of our companion paper, see \cite[Corollary 8.5]{JM-ncld}.  
\end{remark}
\begin{eg} \label{ACeg}
Given any $x \in \hardy$, let $m_x \in \posncm$ denote the positive vector functional:
$$ m_x (L^\alpha ) := \ip{x}{L^\alpha x}_{\hardy}. $$ The results of \cite{JM-F2Smirnov,JM-freeSmirnov} show that given $x \in \hardy$, one can define $x^\mrt (R)=M^R _{x}$, where $x ^\mrt (R)1 = x \in \hardy$ as a densely-defined, closed, and potentially unbounded right multiplier defined on a dense domain, $\nbdom x ^\mrt (R) \supseteq \A$, in the Fock space with symbol in the (right) \emph{free Smirnov class} $\scr{N} _d ^+ (R)$, the set of all ratios of bounded right multipliers $B(R) A(R) ^{-1}$ with outer (dense range) denominator.  We will write $x ^\mrt (R) \sim R^\infty _d$ to denote that $x ^\mrt (R)$ is an unbounded right multiplier affiliated to the right free analytic Toeplitz algebra $R^\infty _d$. That is, it commutes with the left free shifts in the sense that $\nbdom x^\mrt (R)$ is $L-$invariant, and $x^\mrt (R) L_k h = L_k x^\mrt (R) h$ for any $h \in \nbdom x^\mrt (R)$. The (potentially unbounded) left--Toeplitz operator $T := x^\mrt (R) ^* x^\mrt (R)$ is then well-defined, closed, positive semi-definite and densely-defined, and $q_{m_x}$ agrees with $q_T$ on the dense form domain $\A$. The form $q_{m_x}$ is then a closable $L-$Toeplitz form, so that $m_{x}  \in \posncm$ is an absolutely continuous (AC) positive NC measure. In fact, any AC positive NC measure is a vector state on the Fock space (although it may have the asymmetric form $\mu (L^\alpha ) = m_{x,y} (L^\alpha) = \ip{x}{L^\alpha y } _{\hardy}$ for some $y \neq x; \ x,y \in \hardy$) \cite[Remark 6.19, Corollary 6.23]{JM-ncld}. 
\end{eg}

\section{Cauchy Transforms of NC measures}

\label{FreeCTsect}

The goal of this section is to define absolutely continuous and singular NC measures, and to show that any positive NC measure $\mu \in \posncm$ has a unique Lebesgue decomposition, $\mu = \mu _{ac} + \mu _s$, into absolutely continuous and singular parts, $\mu _{ac}, \mu _s \in \posncm$ by proving that the absolutely continuous and singular parts of the Lebesgue form decomposition of $\mu$ are positive NC measures. If $\mu$ is any positive finite and regular Borel measure on $\partial \D$,
the space of all $\mu-$Cauchy transforms,
$$ (\mc{C} _\mu h) (z) := \int _{\partial \D } \frac{1}{1-z\zeta ^*} h(\zeta ) \mu (d\zeta ); \quad \quad h \in H^2 (\mu), \ z \in \D, $$ of the analytic part, $H^2 (\mu)$ of $L^2 (\partial \D , \mu )$, is the reproducing kernel Hilbert space, $\scr{H} ^+ (h _\mu)$ of functions in $\D$ (the Herglotz space of $h_\mu$) with reproducing kernel
$$ K^\mu (z,w) := \frac{1}{2} \frac{ h_\mu (z) + h_\mu (w) ^*}{1-zw^*} = \int _{\partial \D } \frac{1}{1-z\zeta ^*}\frac{1}{1 -\zeta w^*} \mu (d\zeta ), $$ where $h_\mu (z)$ is the Herglotz function of $\mu$.  In this setting, is not difficult to verify that domination of positive measures is equivalent to domination of the kernels for their spaces of Cauchy transforms: 
$$  \mu \leq t^2 \la \quad \Leftrightarrow \quad K^\mu \leq t^2 K^\la. $$ In particular, the following exact NC analogue of a reproducing kernel theory result due to Aronszajn applies \cite[Theorem 5.1]{Paulsen-rkhs} \cite[Theorem I, Section 7]{Aron-RKHS}: 

\begin{thm}
Let $K_1, K_2$ be CPNC kernels on an NC set $\Om$. Then $K_1 \leq t^2 K_2$ for some $t>0$ if and only if  
$$ \mc{H} _{nc} (K_1 ) \subseteq \mc{H} _{nc} (K_2), $$ and the norm of the 
embedding $\mr{e} : \mc{H} _{nc} (K_1 ) \hookrightarrow \mc{H} _{nc} (K_2)$ is at most $t$.
\end{thm}
As in the single-variable setting, it is easy to verify that domination of (positive) NC measures $\mu, \la \in \posncm$ is equivalent to domination of the NC kernels for their spaces of Cauchy Transforms:  
\begin{lemma} \label{CTembed}
Given $\mu , \la \in \posncm$, there is a $t>0$ so that 
$ \mu \leq t^2 \la, $ if and only if $K^\mu \leq t^2 K^\la$. If $E_\mu : \hardy (\la ) \hookrightarrow \hardy (\mu )$ 
and $\mr{e} _\mu : \scr{H} ^+ (H _\mu ) \hookrightarrow \scr{H} ^+ (H _\la )$ are the canonical embeddings defined by 
$$ E_\mu \left( a(L) + N _\la \right) = a(L) + N_\mu, \quad \mbox{and} \quad (\mr{e} _{\mu} h) (Z) = h(Z); \quad Z \in \B ^d _\N, $$ then, $E_\mu = \mc{C} _\mu ^* \mr{e} _\mu ^* \mc{C} _\la$, $\| E _\mu \| = \| \mr{e} _\mu \| \leq t$, and $E_\mu \Pi _\la ^\alpha = \Pi _\mu ^\alpha E_\mu$.  
\end{lemma}
\begin{prop} \label{reduce}
Given any NC measure $\mu \in \posncm$, let $E : \hardy (\mu + m ) \hookrightarrow \hardy$ be the contractive embedding, and let $Q_s$ be the orthogonal projection onto $\nbker E$. Then $Q_s$ is reducing for the GNS row isometry of $\mu + m$.
\end{prop}
\begin{proof}
Since $E \Pi _{\mu +m } ^\alpha = L^\alpha E$, by Lemma \ref{CTembed}, it is easy to check that $Q_s = \nbker E^* E$ is $\Pi _{\mu +m }-$invariant. Indeed, suppose that $x \in \nbker E$ and set $\Pi _k := \Pi _{\mu + m ; k}$ and check that
$$ E \Pi _k x = L_k E x =0. $$ Hence $Q_s$ is $\Pi = \Pi _{\mu +m} -$invariant so that $Q_{ac} = I-Q_s$ is co-invariant and projects onto $\ker{E} ^\perp = \ran{E^*} ^{-\| \cdot \|}$. Let $I - \check{Q} _s := \mc{C} _{\mu +m} (I-Q_s ) \mc{C} _{\mu +m} ^*$ be the projection onto 
$$ \mc{C} _{\mu +m} \ran{E^*} ^{-\| \cdot \| } = \ran{\mr{e}} ^{-\| \cdot \| }, $$ the closure of the range of the contractive embedding,
$$ \mr{e} : H^2 (\B ^d _\N ) \hookrightarrow \scr{H} ^+ (H _{\mu +m} ), $$ in $\scr{H} ^+ (H_{\mu +m })$. Let $\mr{int} (\mu +m , m) = \nbran \mr{e}$, the intersection of the NC Hardy space with the space of Cauchy transforms of $\mu +m$, and let $$\mr{Int} _{\mu +m } (m) := \mr{int} (\mu+m , m) ^{-\| \cdot \| _{H_{\mu+m}}}, $$ the closure of the range of $\mr{e}$. It follows that $\mr{Int} _{\mu +m} (m)$ is $V_{\mu +m}$ co-invariant, and it remains to prove that it is also invariant. Given $f \in \scr{H} ^+ (H _{\mu +m}  ) \bigcap H^2 (\B ^d _\N )$, observe that by Equation \ref{Vmuaction},
\ba (V _{\mu +m; k} f ) (Z) - (V _{\mu +m ; k} f ) (0_n ) & = & Z_k f(Z) \nn \\
& = & (M^L _{Z_k} f) (Z) = (V_{m; k} f) (Z) , \nn \ea so that 
$$ (V _{\mu +m ; k} f ) (Z) = (V _{m ; k} f ) (Z) + c I_n, $$ where $c := (V_{\mu +m; k} f ) (0)$ is constant. 
Since $H^2 (\B ^d _\N )$ contains the constant functions, we conclude that $V_{\mu +m ; k} f \in H^2 (\B ^d _\N ) \bigcap \scr{H} ^+ (H _{\mu +m} )$ also belongs to the intersection space, so that the range of $I -\check{Q} _s$ is reducing for $V_{\mu +m}$, and $Q_{ac} = I-Q_s$, $Q_s$ are reducing projections for the GNS row isometry $\Pi _{\mu +m}$.
\end{proof}
\begin{cor} \label{ACcor}
Given an NC measure $\mu \in \posncm$, the following are equivalent:
\bn
    \item $\mu$ is absolutely continuous.
    \item The intersection of $\scr{H} ^+ (H _{\mu +m} )$ with the NC Hardy space $H^2 (\B ^d _\N )$ is dense in $\scr{H} ^+ (H _{\mu +m} )$. 
\en
\end{cor}
\begin{proof}
By the formula for $q_{ac}$, Equation \ref{qacformula}, and by definition, $\mu$ is absolutely continuous if and only if $Q_{ac} = I _{\mu +m}$, the identity of $\hardy (\mu +m )$. Since $Q_{ac}$ is the orthogonal projection onto the closure of $\nbran E^*$, we have that $Q_{ac} =I$ if and only if $E^*$ has dense range. By Lemma \ref{CTembed}, $E^*$ has dense range if and only if $\mr{e} = \scr{C} _{\mu +m} E^* \scr{C} _m ^* : H^2 (\B ^d _\N ) \hookrightarrow \scr{H} ^+ (H_{\mu +m} )$ has dense range.
\end{proof}
\begin{cor} \label{singcor}
An NC measure $\mu \in \posncm$ is singular if and only if $$ \hardy (\mu +m ) = \hardy (\mu) \oplus \hardy. $$ 
\end{cor}
\begin{proof}
Since $m , \mu \leq  \mu +m$, the embeddings $E : \hardy (\mu +m ) \hookrightarrow \hardy$ and $E_\mu : \hardy (\mu +m ) \hookrightarrow \hardy (\mu )$ are contractions with dense ranges, and $E^* E + E_\mu ^* E_\mu = I _{\mu +m}$, the identity in $\scr{L} (\hardy (\mu +m))$. By the formula of Equation \ref{qacformula}, $\mu$ is singular if and only if $E^*E = P _{\ov{\nbran E^*}} = Q_{ac}$, in which case $Q_s = I - Q_{ac} = I - E^* E = E_\mu ^* E_\mu$. It follows that $E, E_\mu$ are co-isometries onto their ranges, and we can identify $\nbran Q_{ac}$ and $\nbran Q_s$ with $\hardy$ and $\hardy (\mu )$, respectively. 
\end{proof}
\begin{thm}{ (Non-commutative Lebesgue decomposition)} \label{NCLD}
Any positive NC measure $\mu \in \posncm$ has a unique Lebesgue decomposition, 
$$ \mu = \mu _{ac} + \mu _s, $$ into positive NC measures $\mu _{ac}, \mu _s \leq \mu$. The NC measure $\mu _{ac}$ is the maximal absolutely continuous NC measure bounded above by $\mu$, and $\mu _s$ is a singular NC measure. 
\end{thm}
\begin{proof}
Consider the Lebesgue form decomposition $q_\mu = q_{ac} + q_s$ of the positive semi-definite $L-$Toeplitz form, $q_\mu$, with dense form domain $\A \subset \hardy$.
Given $q_{ac}$, we first define a self-adjoint linear functional $\mu _{ac} : \mc{A} _d + \mc{A} _d ^* \rightarrow \C$ by
$$ \mu _{ac} (a_1 ^* + a _2 ) := q_{ac} (a_1, 1) + q_{ac} (1, a_2) ; \quad \quad a_1 \in \mc{A} _d, \ a_2 \in (\mc{A} _d)_0,  $$ where $(\A )_0$ denotes the elements of $\A$ which vanish at $0$, $a_2 (0) =0$. Here, note that we do not take the norm closure of $\mc{A} _d + \mc{A} _d ^*$, so that this is not an operator system. The fact that $\mu _{ac}$ is self-adjoint follows as $q_{ac}$ is positive semi-definite, hence symmetric: Given $a \in \A$, 
$$ \mu _{ac} (a) = q_{ac} \left( 1, (a - a(0) 1) \right) + q_{ac} (\ov{a(0)} 1,1)  = q _{ac} (1, a), $$ and 
$$ \ov{\mu _{ac} (a^*) } = \ov{q_{ac} (a,1)} = q_{ac} (1, a) = \mu _{ac} (a), $$ since $q _{ac}$ is a symmetric sesquilinear form. Observe that $\mu _{ac}$ is bounded on $\A + \A ^*$: Given $a_1 \in \A$, $a_2 \in (\A )_0$, applying the formula of Equation \ref{qacformula} for $q_{ac}$, and using that $Q_{ac}$ commutes with the GNS representation of $\mu +m$ by Proposition \ref{reduce} yields
\ba | \mu _{ac} (a_1 + a_2 ^* ) | & = & \left|  q_{ac} (a_1 , 1) + q_{ac} (1, a_2) \right| \nn \\
& = & \left| \ip{Q_{ac} (I + N _{\mu +m})}{\pi _{\mu +m} (a_1 ^* + a_2 ) Q_{ac} (I + N_{\mu +m} )}_{\mu +m} - \ip{1}{(a_1 (L) ^* + a_2 (L) 1}_{\bH ^2 _d} \right| \nn \\
& \leq & \| \pi _{\mu +m} (a_1 ^* + a_2 ) \| \| Q _{ac} (I + N_{\mu +m} ) \| ^2 + \| a_1 (L) ^* + a_2 (L) \| \nn \\
& \leq & \| a_1 (L) ^* + a_2 (L) \| ( 2 + \mu (I) ). \nn \ea In the above we have used that $\pi _{\mu +m}$ is a unital $*-$representation of the Cuntz-Toeplitz $C^*-$algebra (the unital $C^*-$algebra generated by the left free shifts), and is hence completely positive and completely contractive. This proves that $\mu _{ac}$ is bounded on its domain and hence extends by continuity to a bounded linear functional on the free disk system $\scr{A} _d = (\A + \A ^* ) ^{-\| \cdot \|}$. 

It remains to prove that $\mu _{ac}$ is a positive linear functional.  By \cite[Lemma 4.6]{JM-freeAC}, any positive element of $\scr{A} _d$ is the norm-limit of sums of squares of free polynomials $p_n \in \fp$, and so it suffices to check that $\mu _{ac} (p(L) ^* p(L) )$ is positive semi-definite for any free polynomial $p$. Given any free polynomial of homogeneous degree $N$, 
$$ p(L) := \sum _{|\alpha | \leq N} p_\alpha L^\alpha; \quad \quad p_\alpha \in \C, $$ one can compute that
\ba p (L) ^* p (L) & = & \sum _{|\alpha | \leq N } \sum _{|\beta | \leq N} \ov{p} _\alpha p _\beta (L ^\alpha ) ^* L^\beta  \nn \\
& = & \underbrace{\sum _{\substack{|\alpha |, |\ga | \leq N; \\ \ga \neq \emptyset}} \ov{p_\alpha} p_{\alpha \ga}  L^\ga  + I _{\bH ^2}\cdot \frac{1}{2} \sum _{|\alpha | \leq N} |p_\alpha | ^2}_{=: u(L)} \ + \ u(L) ^*, \label{pu} \ea where $u \in \fp$ has homogeneous degree at most $N$. Then calculate that
\ba \mu _{ac} (p(L) ^* p(L) ) & = & \mu _{ac} (u(L) ^* + u(L) ) \nn \\
& = & q_{ac} ( u(L) , 1 ) + q _{ac} (1 , u (L) ) \nn \\
& = & 2 \nbre \ip{I + N _{\mu +m}}{(Q_{ac} - E^* E) \, (u(L) + N_{\mu +m})}_{\mu+m}. \nn \ea   
Let $\Pi := \Pi _{\mu +m}$. Since $Q_{ac}$ commutes with $\Pi $ by Proposition \ref{reduce}, and $E$ intertwines $\Pi$ with $L$ by Lemma \ref{CTembed}, \emph{i.e.} $E \Pi _{ j} = L_j E$, it follows that the positive semi-definite contraction $D := Q_{ac} - E^* E$ is $\Pi = \Pi _{\mu +m } -$Toeplitz in the sense that
$$ \Pi _{k} ^* D \Pi _{j} = \delta _{k,j} D; \quad \quad 1 \leq k ,j \leq d. $$
This $\Pi -$Toeplitz property allows us to reverse the steps in Equation (\ref{pu}) to conclude that 
$$ u(\Pi  ) ^* D + D u(\Pi ) = p (\Pi ) ^* D p (\Pi ).$$ 
It then follows that
\ba & & 2 \mr{Re} \, \ip{I + N _{\mu +m}}{(Q_{ac} - E^* E) \, (u(L) + N_{\mu +m}) }_{\mu+m} \nn \\ 
& = & \ip{p(L) + N_{\mu +m}}{(Q_{ac} - E^* E) \ (p(L) + N_{\mu +m}) }_{\mu +m} 
 =  q_{ac} (p, p) \geq 0, \nn \ea since $q_{ac}$ is a positive semi-definite sesquilinear form.
Hence $\mu _{ac} \in \posncm$, and by construction 
$$ \mu (p(L) ^* p(L) ) - \mu _{ac} (p(L) ^* p(L) ) = q_{\mu} (p,p) - q_{ac} (p, p) \geq 0. $$ Again, since sums of squares of free polynomials are norm dense in the cone of positive elements in $\scr{A} _d$, $\mu _s := \mu -\mu _{ac} \in \posncm$ and $\mu = \mu _{ac} + \mu _s$. By definition, $\mu _{ac}$ is then an absolutely continuous NC measure, and $\mu _s$ is a singular NC measure.
\end{proof}
\begin{thm} \label{ACiffclosed}
Let $\mu \in \posncm$ be an absolutely continuous NC measure. The closure of the positive semi-definite $L-$Toeplitz form, $q_\mu$, is the sesquilinear form of a unique closed, positive semi-definite $L-$Toeplitz operator, $T$, and $\A$ is a core for $\sqrt{T}$.
\end{thm}
\begin{proof}
If $\mu$ is AC, then by definition, $q_\mu$, with form domain $\nbdom q_{\mu} = \mc{A} _d$ is AC, \emph{i.e.} closable. Hence the form closure $\ov{q _\mu} = q_T$ is the positive semi-definite form of a unique closed, positive semi-definite $L-$Toeplitz operator $T\geq 0$, and by construction, $\mc{A} _d$ is a form-core for $q_T$, hence a core for $\sqrt{T}$. 
\end{proof}

\section{The Radon-Nikodym formula for NC measures} \label{NCRN}
Consider $\mu = \mu _B$ where $B \in \scr{L} _d = [ \mult ] _1$ as before. By Theorem~\ref{NCLD}, $\mu$ has an NC Lebesgue decomposition 
\[
\mu=\mu_{ac}+\mu_s, 
\]
and by Theorem~\ref{ACiffclosed} and Lemma \ref{fpcore}, there is a unique closed, densely-defined, positive $L$-Toeplitz operator $T$ such that $\A\subset \mr{Dom} \sqrt{T} $,
\be
\mu_{ac}(a_2^*a_1) = \langle \sqrt{T}a_1, \sqrt{T}a_2\rangle_{\hardy} = q_T (a_1 , a_2 ); \quad \quad a_1, a_2 \in \mc{A} _d,
\ee
$\fp \subseteq \mr{Dom} \sqrt{T}$ is a core for $\sqrt{T}$, and $\mr{Dom} \sqrt{T}$ is $L-$invariant. The operator $T$ can be thought of as the NC Radon-Nikodym derivative of $\mu$ with respect to NC Lebesgue measure, $m$.
Our goal is to recover $T$ from the function $B$, or, more precisely, from the net of scaled right multipliers $B_r(R):=B(rR)$, $0<r<1$.  For each $0<r<1$ we form the (bounded, positive, invertible, $L-$Toeplitz) operator $T_r$:
 \begin{align}
 T_r &:= \nbre H_B (rR)  > 0 \nn \\
 &= (I-B(rR)^* ) ^{-1} \left( I- B(rR)^* B(rR) \right) (I - B(rR) ) ^{-1}. \label{Trform} 
 \end{align}
\begin{remark} \label{strictcon}
Here, $B(L) \in [L^\infty _d ] _1$ is a contraction, and so $B(rL)$ is a strict contraction, 
$$ \| B (rL) \| = \sup _{Z \in \B ^d _\N } \| B (rZ) \|  = \sup _{Z \in r \B ^d _\N} \| B(Z) \| <1, $$ unless $B(Z) = \beta I_n$, $\beta \in \partial \D$, is identically constant, by \cite[Theorem 3.1]{SSS} and the NC maximum modulus principle of \cite[Lemma 6.11]{SSS}. Hence, $B(rR) = U_\mrt B(rL) U_\mrt$ is also a strict contraction for any $0<r<1$.
\end{remark}
 We observe that in one variable, the quadratic form induced by the Toeplitz operator $T_r$ is 
\[
\langle T_r f, g\rangle_{H^2} = \int_{\partial \D } f(\zeta)\overline{g(\zeta)} \frac{1-|b(r\zeta)|^2}{|1-b(r\zeta)|^2}\, m(d \zeta); \quad \quad  f, g\in H^2,
\]
and by Fatou's theorem, we have that the Radon-Nikodym derivative of $\mu_b$ is: 
\[
\frac{\mu_b (d\zeta )}{m (d\zeta )} =  \frac{1-|b(\zeta)|^2}{|1-b(\zeta)|^2} =\lim_{r\uparrow 1}  \frac{1-|b(r\zeta)|^2}{|1-b(r\zeta)|^2}.
\]
Thus $T$ can be recovered from the $B_r$ by forming the Toeplitz operators $T_r$ and taking a.e. limits of the symbols, so that $T$ is the densely defined Toeplitz operator with $L^1$ symbol \be \frac{1-|b(\zeta)|^2}{|1-b(\zeta)|^2}; \quad \quad \zeta \in \partial \D. \label{symbol} \ee

In the non-commutative setting, we seek to recover the `NC Radon-Nikodym derivative', $T$, from its radial values, $T_r$, by re-casting the original Fatou Theorem in purely operator-theoretic or functional analytic terms. One technical hurdle is that while the operators $T_r$ are bounded, $T$ generically is not. It turns out that the notion of {\em strong resolvent convergence} (\emph{i.e.} strong operator topology convergence of resolvents) of positive semi-definite operators is appropriate.  This notion is defined in Subsection~\ref{subsec:SR-def}. The main goal of this section is to prove:

\begin{thm}[Non-commutative Fatou Theorem] \label{main}
Let $\mu = \mu _B \in \posncm$ be the NC Clark measure of $B \in [\mult ] _1$ with NC Lebesgue decomposition $\mu = \mu _{ac} + \mu _s$. If $T \geq 0$ is the positive semi-definite $L-$Toeplitz operator so that $q_T = \ov{q_{\mu_{ac}}}$, and $T_r := \nbre H_B (rR)$, $0<r<1$, as above then  $T_r$ converges to $T$ in the strong resolvent sense as $r \uparrow 1$. 
\end{thm}
Using this theorem (and without appeal to Fatou's original result) it is not too difficult to show, in the one variable case, that $T$ is the (possibly unbounded) Toeplitz operator with the symbol given above in Equation (\ref{symbol}),
as of course must be the case.

\subsection{Radial approximation of NC Herglotz functions}

Let $B \in \scr{L} _d = [H^\infty (\B ^d _\N ) ] _1$ be a contractive NC function in the NC unit ball, and let $B^{\mr{t}} \in \scr{R} _d$ be its transpose-conjugate in the right free Schur class.  For any $0<r<1$, define  $$ B _r (Z) := B (rZ), \quad  B_r (R) = B(rR) = M^R _{B^\mrt _r}, \quad \mbox{and} \quad \mu _r := \mu _{B_r} \in \posncm. $$
\begin{lemma} \label{sotstar}
The net $B(rR)$ converges $SOT-*$ to $B(R)$ as $r \uparrow 1$. 
\end{lemma} 
\begin{proof} 
One can show, as in \cite[Theorem 3.5.5]{Sha2013}, that $B_r (R) \stackrel{SOT}{\rightarrow} B(R)$. In fact, $B_r (R) \stackrel{SOT-*}{\rightarrow} B(R) = M^R _{B^{\mr{t}} (Z)}$, since for any kernel function $K \{Z , y ,v \}$, 
\ba & & \| (B_r (R) - B (R) ) ^*  K \{ Z , y , v \} \| ^2  =   \| K \{ Z , y , (B^{\mr{t}} (rZ) - B ^{\mr{t}} (Z) ) v \} \| ^2 \nn \\ & = & \ipcn{y}{K(Z,Z)[(B^{\mr{t}} (rZ) - B ^{\mr{t}} (Z) )vv^* (B ^{\mr{t}} (rZ) - B ^{\mr{t}} (Z) ) ^*] y} \rightarrow 0. \nn \ea
Using that the set of all kernel functions is dense in $\hardy$, and that the net $B_r (R)$ is uniformly norm-bounded, it follows that $B(rR) ^*$ converges $SOT$ to $B(R) ^*$.
\end{proof}

\begin{lemma} \label{rapprox}
    $\mu _r = \mu _{B_r}$ converges pointwise to $\mu _B$ as $r \uparrow 1$. That is, for any $a_1, a_2 \in \A$, $\mu _r (a_1 ^* a_2) \rightarrow \mu _B (a_1 ^* a_2)$. 
\end{lemma}
\begin{proof}
This follows from the formula for the NC Clark measure of Definition \ref{NCClarkdef}, and that $H_B ^{\mr{t}} (rZ)$ converges pointwise to $H_B ^{\mr{t}} (Z)$ as $r \uparrow 1$. (Or, equivalently, one can apply the Taylor-Taylor series formula for $H_B (rZ)$ of Equation \ref{HergTTseries}.)
\end{proof}

\subsection{Strong resolvent convergence of self-adjoint operators}\label{subsec:SR-def}

It will be useful to recall some basic facts about convergence of unbounded self-adjoint operators. Our main reference for this is \cite[Chapter VIII.7]{RnS1}:
\begin{defn}
Let $A _n , A$ be closed, self-adjoint operators. The sequence $A _n$ converges to $A$ in the strong resolvent (SR) sense if $(A _n -\la I ) ^{-1} \rightarrow (A  - \la I ) ^{-1}$ in the strong operator topology (SOT) for any $\la \in \C \sm \R$. 
\end{defn}
To show strong resolvent convergence it suffices to check $SOT$ convergence of the resolvents at any fixed point $\la \in \C \sm \R$:
\begin{thm}{ (\cite[Theorem VIII.19]{RnS1})}
Let $A _n , A $ be self-adjoint operators. Then $A _n \stackrel{SR}{\rightarrow} A$ if and only if there is a $\la _0 \in \C \sm \R$ so that $(A _n - \la _0 I ) ^{-1} \stackrel{SOT}{\rightarrow} (A -\la _0 I ) ^{-1}$.
\end{thm}

\begin{remark}
    We will be primarily interested in positive semi-definite $A _n , A$, in which case one can instead choose $\la _0 \in (-\infty , 0 )$. In particular, it suffices to show $(I+A _n ) ^{-1} \stackrel{SOT}{\rightarrow} (I + A ) ^ {-1}$.
\end{remark}

\subsection{Strong resolvent convergence of positive free harmonic functions}
Recall that we have defined
$$ 
T_r =(I-B(rR)^* ) ^{-1} \left( I- B(rR)^* B(rR) \right) (I - B(rR) ) ^{-1} > 0; \quad  0<r<1. $$ 
To investigate the strong resolvent convergence of the $T_r$, we define $\Delta _r (\eps ) := ( \eps I +T _r ) ^{-1}$ for any $\eps >0$. In particular, if $\Delta _r := \Delta _r (1)$, then 
$$ \Delta _r  =  (I + T_r ) ^{-1} = \frac{1}{2} (I - B(rR) ) \left( I - \nbre B(rR) \right) ^{-1} (I - B(rR )^*), $$ Since the spectrum, $\sigma (I + T_r ) \subseteq [1, \infty )$, it follows that $0 \leq \Delta _r \leq I$ for any $0<r<1$. Also note that $r \mapsto \Delta _r$ is $SOT-$continuous for $r \in (0,1)$.  
 
Our strategy will be to prove that if $\Delta (\eps ) := (\eps I +T ) ^{-1}$, where $T \geq 0$ is the closed, positive semi-definite NC Radon-Nikodym derivative of $\mu = \mu _B$ with respect to NC Lebesgue measure, \emph{i.e.} $\ov{q_{\mu _{ac}}} = q_T$, then 
$$ \Delta _r (\eps ) \stackrel{WOT}{\rightarrow} \Delta (\eps ) = (\eps I +T) ^{-1}, $$ where $WOT$ denotes weak operator topology. Using the resolvent formula, it will then follow that $\Delta _r = \Delta _r (1) \stackrel{SOT}{\rightarrow} \Delta$, so that $T_r$ converges to $T$ in the strong resolvent sense.

\subsection{Factorization of unbounded $L-$Toeplitz operators}

Any strictly positive $L-$Toeplitz operator which is bounded above and below has an analytic outer factorization:
\begin{thm}{ (Popescu \cite[Theorem 1.5]{Pop-entropy})} \label{Popfactor}
Any positive $L-$Toeplitz $T \in \scr{L} (\hardy )$ which is bounded below, $T \geq \eps I$ can be factored as:
$T = F (R) ^* F (R)$ for some outer $F (R) \in R^\infty _d $.  
\end{thm}
Here, recall that $F(R) \in R^\infty _d$ is called outer if it has dense range.
The goal of this subsection is to extend this factorization result to closed, unbounded $L-$Toeplitz operators. 
\begin{thm} \label{tfactor}
Let $\tau \geq 0$ be a closed, positive semi-definite $L-$Toeplitz operator so that $\A$ is a core for $\sqrt{\tau}$. Then, for any fixed $\eps >0$ there is a right-outer $x _{[\eps ]} (R) \sim R^\infty _d$, $\nbdom x _{[\eps ]} (R)  = \nbdom \sqrt{\tau }$ so that $x _{[\eps ]} (R) ^* x _{[\eps ]} (R) = \eps I + \tau$, and $x _{[\eps ]} ^{\mrt}  = x _{[\eps ]} (R) 1 \in \hardy$. The closed right multiplier $x_{[\eps ]} (R)$ is unique up to a unimodular constant.  
\end{thm}
In the above, recall that the notation $x(R) \sim R^\infty _d$ means that $x(R)$ is a potentially unbounded, closed right multiplier affiliated to the operator algebra $R^\infty _d$, see Example \ref{ACeg}. That is, $\nbdom x(R)$ is $L-$invariant and $x(R) L_k h = L_k x(R) h$ for any $h \in \nbdom x(R)$. Any such unbounded multiplier can be expressed as $x(R) = B(R) A(R) ^{-1}$ with $B(R) , A(R) \in R^\infty _d$ bounded right multipliers so that $A(R)$ is NC outer, \emph{i.e.} has dense range \cite{JM-freeSmirnov}. The above theorem is essentially \cite[Theorem 1.1, Corollary 1.4]{Pop-entropy}, without the assumption that $\tau $ is bounded. We pause to observe that in one variable this result is a familiar consequence of the theory of outer functions: If $h$ is a non-negative $L^1$ function on the circle, (think of this as the symbol of $\tau $), then $1+h$ is log-integrable on the unit circle $\partial \D  $, hence there is an $H^2$ outer function $g$ such that $1+h=|g|^2$ almost everywhere on $\partial \D$ \cite[Chapter IV]{Hoff}, or, (ignoring technicalities about domains)
\[
I+T_h =T_{1+h} =T_g^*T_g. 
 \]  (Of course $h$ itself need not be log-integrable and hence need not factor.) 
\begin{defn}
Let $\Pi$ be a row isometry on a separable Hilbert space $\cH$. A closed, densely-defined operator $X : \nbdom X \subseteq \hardy \rightarrow \cH$ is called an \emph{intertwiner} if $\nbdom X$ is $L-$invariant and 
$$ X L_k x = \Pi _k X x; \quad \quad x \in \nbdom X. $$
\end{defn}
\begin{thm} \label{NoCuntz}
Let $\Pi$ be a cyclic row isometry on $\cH$. Suppose $X : \nbdom X \subseteq \hardy \rightarrow \cH$ is an intertwiner which is densely-defined, closed, surjective, and bounded below. Then $\Pi$ is unitarily equivalent to $L$.  
\end{thm}
Any $X$ satisfying the above theorem has a bounded inverse, by the open mapping theorem. This result is essentially \cite[Theorem 2.8]{DLP-ncld}, without the assumption that $X$ is bounded. Our proof follows the same lines, though the unbounded case requires some care. 
\begin{proof}
The Cuntz part of $\Pi$ is supported on:
\ba \mc{K} &:=& \bigcap _{k=1} ^\infty \left( \bigvee _{|\alpha | = k } \nbran \Pi ^\alpha \right) \nn \\
& = & \bigcap _{k=1} ^\infty \left( \bigvee _{|\alpha | = k } \Pi ^\alpha X \nbdom X \right) \quad \quad \mbox{(by surjectivity of $X$)} \nn \\
& = &  \bigcap _{k=1} ^\infty \left( \bigvee _{|\alpha | = k } X L ^\alpha \nbdom X \right) \nn 
\ea

We now prove that the space $ \bigcap _{k=1} ^\infty \left( \bigvee _{|\alpha | = k } X L ^\alpha \nbdom X \right)$ is $\{0\}$.  Let
\[
\mathcal M_k = \bigvee _{|\alpha|=k} L^\alpha \nbdom X, \quad k=1, 2, \dots
\]
(here we are taking just the linear span, not its closure).   We have $\mathcal M_k\subset \nbdom X$ for each $k$, and 
\[
\bigcap_{k=1}^\infty \overline{\mathcal M_k} \subset\bigcap_{k=1}^\infty \left( \bigvee _{|\alpha|=k} L^\alpha \hardy \right) =\{0\}. 
\]
Now suppose that $h\in\mathcal{K} = \bigcap_{k=1}^\infty X\mathcal M_k$, we want to prove $h=0$. For each $k$ there is a $g_k\in \mathcal M_k$ such that $h=Xg_k$. Since $X$ is assumed bounded below, we have
\[
\|h\|=\|Xg_k\| \geq c\|g_k\|
\]
for some absolute constant $c>0$, and all $k$. Thus the sequence $g_k$ is uniformly bounded, so there is a subsequence $(g_{n_k})$ converging weakly to some $g\in \hardy$. Since the $\mathcal M_k$ are nested, this $g$ must belong to the intersection of all the $\overline{\mathcal M_k}$, which we have just observed is $\{0\}$. Thus $g=0$. Let now $f\in \nbdom X^*$, which is a dense subspace since $X$ is closed. We have
\[
\langle h,f\rangle = \langle Xg_{n_k}, f\rangle =\langle g_{n_k}, X^*f\rangle \to 0.
\]
Since this holds for all $f\in \nbdom X^*$, we conclude $h=0$. We have shown that $\mathcal K=\{0\}$, hence $\Pi$ is of pure type$-L$, \emph{i.e.} $\Pi$ is unitarily equivalent to copies of $L$. Since $\Pi$ is cyclic, its wandering space is one-dimensional, and $\Pi$ is unitarily equivalent to $L$. 

\end{proof}

\begin{proof}[Proof of Theorem \ref{tfactor}]
Without loss in generality assume $\eps =1$ and consider $\sqrt{I + \tau  }$. This is closed, bounded below, strictly positive, and hence bijective in $\hardy$ (its range is all of $\hardy$ and it is injective on its dense domain).  Both $\sqrt{\tau }$ and $\sqrt{I+\tau }$ are (generally unbounded) closed positive semi-definite operators, and it follows from spectral theory that $\nbdom \sqrt{\tau } = \nbdom \sqrt{I+\tau }$. It is further easy to check that the norms
$$ \| h \| ^2 _{G (\sqrt{\tau })} := \| h \| ^2 + \| \sqrt{\tau } h \| ^2, \quad \mbox{and} \quad 
\| h \| ^2 _{I + \tau } := \| \sqrt{I + \tau } h \| ^2, $$ coincide on $\nbdom \sqrt{\tau }$. (Indeed, it is clear that they coincide on $\nbdom \tau $, which is a core for both $\sqrt{\tau }$ and $\sqrt{I+\tau}$.)   Set $X := \sqrt{I + \tau}$ and define a row isometry, $\Pi$, on $\nbran X = \hardy$ by the equation:
$$ \Pi _k X y := X L_k y; \quad \quad y \in \nbdom X. $$ 
As in Equation \ref{Trowisom}, $\Pi$ is a row isometry since $q_{I +\tau } = q_\tau + q_m$ (where recall $m$ denotes the Lebesgue vacuum state) is an $L-$Toeplitz form.
Recall that since $\tau$ is a closed, positive semi-definite $L-$Toeplitz operator with $\A$ a core for $\sqrt{\tau}$, that $\nbdom \sqrt{\tau}$ is $L-$invariant and contains the free polynomials as a core by Lemma \ref{fpcore}. The free disk algebra is also a core for $\sqrt{I+\tau}$: Recall that $\A$ is a core for $\sqrt{\tau}$ or $\sqrt{I+\tau}$ if and only if it is a form core for the closed forms $q_\tau$, $q_{I+\tau}$, respectively, see Subsection \ref{Closedforms}. By definition, $\A$ is a form core for $q_\tau$, or $q_{I+\tau}$ if and only if $\A$ is dense in $\cH (q_{\tau} +1)$ or $\cH (q_{I+\tau} +1)$ respectively. The norms of $\cH (q_\tau +1), \cH (q_{I+\tau} +1)$ are equivalent and therefore $\A$ is a core for $\sqrt{I+\tau}$. In conclusion, Lemma \ref{fpcore} implies that the free polynomials are a core for $\sqrt{I+\tau}$, $\nbdom \sqrt{I+\tau}$ is $L-$invariant and $X$ is a closed intertwiner. Since $\fp$ is a core for $X=\sqrt{I +\tau}$, any $h = \sqrt{I + \tau} g \in \nbran X$ is the norm-limit of vectors of the form $X p_n (L) 1$, for $p_n \in \fp$. Hence,
$$ h = \lim _{n \rightarrow \infty} p_n (\Pi ) X 1, $$ so that $X1$ is cyclic for $\Pi$.
In summary, $X : \nbdom X \rightarrow \hardy$ is a densely-defined, closed and bijective intertwiner, and $\Pi$ is a cyclic row isometry so that Theorem \ref{NoCuntz} implies that $\Pi \simeq L$ is unitarily equivalent to $L$ via a unitary $U$ on $\hardy$, $U ^* \Pi U = L$.  The operator $U^* X$ then commutes with the left free shifts: 
$$ U^* X L_k h = U^* \Pi _k X h = L_k U^* X h; \quad \quad h \in \nbdom X, $$ and it follows that $U^* X = x(R) \sim R^\infty _d$ is a right-Smirnov multiplier affiliated to the right multiplier algebra of the NC Hardy Space \cite[Corollary 4.26]{JM-freeSmirnov}, and $$ x(R) ^* x(R) = X^* X = I +\tau. $$  Since $x(R) = U^* \sqrt{I+\tau }$, the free polynomials are a core for $x(R)$, and $\nbran x(R)$ is dense in $\hardy$ so that $x(R)$ is right-outer. Since $1 \in \nbdom \sqrt{I+\tau} = \nbdom x(R)$ it follows that $x^\mrt= x(R)1 \in \hardy$ is an $L-$cyclic vector. To prove uniqueness, suppose that $y(R) = M^R _{y ^\mr{t}} \sim R^\infty _d$ is an outer right multiplier so that $x(R) ^* x(R) = I + \tau = y(R) ^* y(R)$. Since both $x(R), y(R)$ are outer (\emph{i.e.} they have dense range), $x(R)^*, y(R)^*$ are injective. Also, any $x(R)  \sim R^\infty _d$ is necessarily injective by \cite[Lemma 4.9, Corollary 4.26]{JM-freeSmirnov}. By the polar decomposition for closed operators, \cite[Proposition 5.3.18]{analysisnow}, it follows that there are unitary operators $U,V$ so that 
$$ x(R) = U \sqrt{I+\tau} , \quad \mbox{and} \quad y(R) = V \sqrt{I+\tau }, $$ and hence $x(R)= UV^* y(R)$. Then, for any $h \in \nbdom y(R)$, 
$$ \left( L_k x(R) - x(R) L_k \right) h =0 \ \Rightarrow \ (L_k UV^* - UV^* L_k ) y(R) h =0. $$ Since $y(R)$ has dense range, we conclude that $UV^*$ commutes with the left shifts so that $UV^* = W (R) \in R^\infty _d$ is a unitary right multiplier \cite[Theorem 1.2]{DP-inv}. However, the only normal elements of $R^\infty _d$ are scalar multiples of the identity, and it follows that $W (R) = \zeta I$, $\zeta \in \partial \D$ is a unitary constant \cite[Corollary 1.5]{DP-inv}.
\end{proof}
\begin{remark}
In the above, since $I+\tau \geq I$ is bounded below by $1$ (under the assumption that $\eps =1$), it follows that $x(R) ^{-1}$ is a contractive right multiplier.
\end{remark}

\begin{remark}
Theorem \ref{tfactor} can be readily extended to general positive left Toeplitz operators, $\tau >0$, such that $\nbdom \sqrt{\tau}$ contains an $L-$cyclic vector, $h \in \hardy$.
\end{remark}

\begin{lemma} \label{SRlimlem}
For any $\eps >0$, the bounded operators $x _{[\eps ]} (rR) ^{-1} x _{[\eps ]} (rR) ^{-*}$ converge in the strong operator topology to $(\eps I +\tau ) ^{-1}$ as $r\uparrow 1$. 
\end{lemma}
In the above, and throughout, we employ the notation $ (A^*) ^{-1} =: A^{-*}$.
\begin{proof}
    Immediate, by Lemma \ref{sotstar}.
\end{proof}
\begin{thm}
Let $\tau \geq 0$ be any closed, positive semi-definite $L-$Toeplitz operator so that $\A$ is a core for $\sqrt{\tau}$. If $\eps >0$ is fixed and $y_{[\eps ]} (R) \sim R^\infty _d$ is the potentially unbounded outer right multiplier so that $\eps I+\tau = y _{[ \eps ]} (R) ^* y _{[ \eps ]}  (R)$ and $y _{[\eps ]} ^\mrt = y _{[\eps ]} (R) 1 \in \hardy$, then $y _{[\eps ]} (R) = (M^R _{\psi _\eps} ) ^{-1}$ where 
$$ \psi _\eps  = \frac{1}{\sqrt{\ip{1}{(\eps I+\tau) ^{-1} 1}}} \cdot (\eps I+\tau) ^{-1} 1. $$
\end{thm}
\begin{proof}
This is an extension of \cite[Theorem 1.5]{Pop-entropy} to closed, unbounded $L-$Toeplitz operators, and the proof is similar. Again, without loss in generality we assume that $\eps =1$. Let $\varphi := (I+\tau) ^{-1} 1$. We first claim that 
$M^R _\varphi $ defines a bounded operator in $R^\infty _d$. To see this note that $\varphi \in \dom{I+\tau} = \nbdom y(R) ^* y(R) \subset \nbdom y(R)$. Hence $y(R) \varphi \in \nbdom y(R) ^*$ and for any $\alpha \in \F ^d$, $ L^\alpha y(R) \varphi = y (R) L^\alpha \varphi \in \nbdom y (R) ^*$ since $\nbdom y (R)^*$ is $L-$invariant by \cite[Corollary 4.27]{JM-freeSmirnov} and also $L^\alpha \varphi  \in \nbdom y (R)$. We conclude that 
$L^\alpha \varphi \in \nbdom y(R) ^* y (R) $ for any $\alpha \in \F ^d$. Given any free polynomial, $ p (L ) := \sum _\alpha p_\alpha L^\alpha \in \C \{ L_1 , \cdots , L_d \}$,  we have that $p(L) \varphi \in \dom{I+\tau}$ and 
\ba \| M^R _\varphi p \| ^2 & = & \| p(L) \varphi \| ^2 \nn \\
& \leq & \| \sqrt{I+\tau} p(L) \varphi \| ^2 \quad \quad \quad \quad \quad \quad \mbox{well-defined as $\nbdom \sqrt{I+\tau} = \nbdom y (R)$,} \nn \\
& = & \sum _{\alpha ,\beta } \ov{p _\alpha } p _\beta  \ip{L^\alpha \varphi}{(I+\tau) L^\beta \varphi} \nn \\
& = & \sum _\alpha |p _\alpha | ^2 \cdot \ip{\varphi}{(I+\tau) \varphi}  + \sum _{\ga \neq \emptyset, \alpha} \ov{p_\alpha} p_{\alpha \ga} \ip{\varphi}{(I+\tau) L^\ga \varphi} +\  \mbox{c.c.} \nn \\
& = & \sum _\alpha |p _\alpha | ^2 \cdot \ip{(I+\tau) ^{-1} 1}{(I+\tau) (I+\tau ) ^{-1} 1} + 
\sum _{\ga \neq \emptyset, \alpha} \ov{p_\alpha} p_{\alpha \ga} \underbrace{\ip{(I+\tau ) ^{-1} 1}{(I+\tau) L^\ga \varphi}}_{=0 \ \mbox{(as $\ga \neq \emptyset$)}} + \ \mbox{c.c.} \nn  \\
& =& \sum _\alpha |p _\alpha | ^2 \cdot \ip{1}{(I+\tau) ^{-1} 1} = \| p \| ^2 _{\hardy} \ip{1}{(I+\tau) ^{-1} 1}_{\hardy}, \nn \ea where \textrm{c.c.} denotes complex conjugate of the previous term. This proves that $M^R _\varphi$ extends to a bounded operator. Moreover, a similar calculation shows that for any $p, q \in \fp$,
$$ \ip{(M^R _\varphi ) ^* (I+\tau) M^R _\varphi p}{q} = \ip{p}{q} _{\hardy} \ip{1}{(I+\tau) ^{-1} 1} _{\hardy}, $$ 
so that $$ (M^R _\varphi ) ^* (I+\tau) M^R _\varphi = I\ip{1}{(I+\tau) ^{-1} 1}, $$ or, defining $\psi $ as in the theorem statement, 
\be (M ^R _\psi) ^* (I+\tau) M^R _\psi = I. \label{identity} \ee  We claim that $M^R _\psi =: F(R) \in R^\infty _d$ is outer, namely that $\psi$ is cyclic for $L$. To see this suppose that $h \in \hardy$ is orthogonal to $\bigvee L^\alpha \psi$, or equivalently, to $\bigvee L^\alpha \varphi$. Since $(I+\tau) ^{-1}$ is bounded, the closed operator $I+\tau$ is surjective which means that $h = (I+\tau) g $ for some $g \in \hardy$. Then, 
$$ (I+\tau) g \perp L^\alpha (I+\tau) ^{-1} 1, $$ for any $\alpha \in \F ^d$. In particular, taking $\alpha = \emptyset$ shows $g_\emptyset =0$ so that $g = L L ^* g$ Hence, 
\ba 0 & = & \ip{(I+\tau) L L ^* g }{L_k \varphi} _{\bH ^2 _d} \nn \\
& = & \ip{\sqrt{I+\tau} L L^* g}{ \sqrt{I+\tau} L_k \varphi }_{\bH ^2 _d } \nn \\
& = & \ip{\sqrt{I+\tau} L_k ^* g}{\sqrt{I+\tau} \varphi}_{\bH ^2 _d} \nn \\
& = & \ip{L_k ^* g}{(I+\tau) \varphi} = \ip{ g}{L_k 1}, \nn \ea and this shows that all Taylor coeficients of order 1 of $g$ also vanish. Repeating this argument shows that all Taylor coefficients vanish so that $g \equiv 0$, both $\varphi, \psi$ are $L-$cyclic, and $F(R) = M^R _\psi$ is right-outer (and hence has a right-Smirnov inverse affiliated to $R^\infty _d$ \cite{JM-F2Smirnov,JM-freeSmirnov}). 

By previous calculation we have $(M^R _\psi ) ^* y(R) ^* y(R) M^R _\psi = I$, and it follows that $y(R) M^R _\psi \in R^\infty _d$ is a bounded, unitary right multiplier. However, the only normal elements of $R^\infty _d$ are scalar multiples of the identity \cite[Corollary 1.5]{DP-inv}, so that $y(R) M^R _\psi = \alpha I$ for some $\alpha \in \partial \D$, and we can assume, without loss in generality, that $\alpha =1$ and $M^R _\psi = y(R) ^{-1}$. Finally, by Equation (\ref{identity}), $y(R) ^{-*} (I + \tau ) y(R) ^{-1} =I$, so that $I+\tau = y(R) ^* y(R)$. Here, recall our notation: $y(R) ^{-*} := (y(R) ^* ) ^{-1}$.
\end{proof}
\begin{remark} \label{xformula}
Taking $\tau =T \geq 0$ to be the NC Radon-Nikodym derivative of $\mu = \mu _B$ with respect to NC Lebesgue measure, \emph{i.e.} $q _T$ is the closure of $q _{\mu _{ac}}$, the above results can be applied to any $\Delta  (\eps ) = (\eps I + T ) ^{-1}$ for any $\eps >0$. In particular, $\Delta = \Delta (1) = (I +T) ^{-1}$ factors (uniquely up to $\zeta \in \partial \D$) as:
\be \Delta = x(R) ^{-1} x(R) ^{-*}, \quad \mbox{and} \quad I+T = x(R) ^* x(R), \label{Deltafact} \ee where $x(R) ^{-1} \in [R^\infty _d ] _1$ is contractive, right-outer, $x(R) \sim R^\infty _d$, and $x(R) 1 = x^\mrt \in \hardy$ is $L-$cyclic.   
\end{remark}

\subsection{Radial approximation of free harmonic functions}

For any $0<r<1$ consider the map $\Phi _r : R^\infty _d \rightarrow R^\infty _d$ defined by:
$$ \Phi _r (R_k ) := r R_k. $$ Since $rR$ is a strict row contraction, it follows that $\Phi _r$ extends to a completely positive and unital map on the (right) free Toeplitz system $\left( R^\infty _d + (R^\infty _d ) ^* \right) ^{-weak-*}$ \cite[Corollary 2.3]{Pop-NCdisk}.

Observe that $T$ has well-defined Taylor coefficients:
$$ T _\alpha := \ip{\sqrt{T} L^\alpha 1}{\sqrt{T} 1} _{\hardy}, $$ 
and that by definition, 
$$ T _\emptyset = (H _{\mu _{ac}} ) _\emptyset, \quad \mbox{and} \quad T _\alpha = \frac{1}{2}(H _{\mu _{ac}} ) _\alpha, $$ for $\alpha \neq \emptyset$ are (up to a constant) the Taylor coefficients of the NC Herglotz function $H_{\mu _{ac}}$ since $q_T $ is the closure of $q _{\mu _{ac}}$. It follows that we can define the bounded $L-$Toeplitz operators $\Phi _r (T)$ by the Taylor coefficients:
$$ \left( \Phi _r (T) \right)  _{\alpha} = T _\alpha r ^{| \alpha |} = \frac{1}{2} (H _{\mu _{ac}} ) _\alpha r ^{|\alpha |}, $$ and it follows that 
$$ q _{\Phi _r (T)} =  ( \mu _{ac} ) _r, $$ as defined in Lemma \ref{rapprox} so that 
$$ 0 \leq \Phi _r (T )  = \nbre H _{\mu _{ac}} (r R) \leq \nbre H _\mu (rR) = T_r. $$

In the disk, if $h$ is a bounded, positive harmonic function and we have $h(\zeta)= |g(\zeta)|^2$ on the circle, for some bounded analytic $g$, then $|g(z)|^2\leq h(z)$ inside the disk, since $h$ is harmonic, $|g|^2$ is subharmonic, and they have the same boundary values. An NC version of this, adapted to our purposes, is the following: 
\begin{lemma} \label{subharlem}
For any $0<r<1$, we have the harmonic majorant inequality:
\be x(rR) ^* x(rR) \leq I + \Phi _r (T) \leq I + T _r. \label{NCsubharm} \ee 
\end{lemma}
In the above statement, recall that $x(R) \sim R^\infty _d$ is as defined in Remark \ref{xformula}, $x(R) ^* x(R) = I+T$.
\begin{proof}
By the Schwarz inequality for unital $2-$positive maps \cite[Proposition 3.3]{Paulsen}, for any $a(R) \in R^\infty _d $, it follows that 
$$ a(rR) ^* a(rR) \leq \Phi _r (a(R) ^* a(R) ), $$ and we need to show that this inequality holds for potentially unbounded right multipliers, $x(R) \sim R^\infty _d$.

Let $x_N ^\mr{t}$ be the $N$th partial sum of $x^\mrt = x(R) 1 = \sum x_{\alpha ^\mrt} L^\alpha 1$, $x_N ^\mrt = \sum _{|\alpha | \leq N} x_{\alpha ^\mrt} L^\alpha 1$, so that $x_N ^\mrt \rightarrow x ^\mrt$ in $\hardy$, and define
$$ I + T _N := x_N (R) ^* x_N (R). $$ 
Observe that for any $\alpha \in \F ^d$,
\begin{align*} \lim _{N \rightarrow \infty} \left( \Phi _r (I + T _N ) \right) _\alpha & =  |r| ^{\alpha } \lim (I + T _N ) _\alpha  \\
& =  |r| ^\alpha \lim  \ip{L^\alpha 1}{x_N (R) ^* x_N (R) 1}  \\
& =  |r | ^\alpha \lim \ip{L^\alpha  x_N ^\mrt }{x_N ^\mrt }  \\
&=  |r | ^\alpha  \ip{L^\alpha x ^\mrt }{x ^\mrt }  \\
& =  |r | ^\alpha \ip{x(R) L^\alpha 1}{x(R) 1}  \\
& =  |r| ^\alpha \ip{\sqrt{I+T} L^\alpha 1}{\sqrt{I+T} 1}  \\
& =  |r| ^\alpha ( I + T ) _\alpha = (I + \Phi _r (T )  ) _\alpha.  \end{align*}
It follows that if 
$$ p(L) = \sum _{|\alpha | \leq M} p_\alpha L^\alpha, $$ is any free polynomial, then 
\ba & &  \lim _{N\rightarrow \infty } \ip{p}{\Phi _r \left( I + T _N  \right) p}_{\hardy} \nn \\ 
&  & = \lim \sum _{\alpha } | p _\alpha | ^2 (I +T _N ) _\emptyset + \lim \sum _{\substack{\alpha, \ga  \\ \ga \neq \emptyset}} \ov{p_\alpha } p_{\alpha \ga} \underbrace{\ip{1}{ \Phi _r (I +T _N )  L^\ga 1 } }_{=r^{|\ga | } \ov{(I+T _N )_\ga}}
+ \mr{c.c.} \nn \\
&  & = \ip{p}{(I+ \Phi _r (T) )p}_{\hardy}, \label{formula} \ea by the previous calculation. In the above $\mr{c.c.}$ denotes complex conjugate of the previous term. Hence,
\ba \ip{x(rR) p}{x(rR) p} & = & \| p (L) x(rR) 1 \| ^2 \nn \\
& = & \lim _N \| p(L) x_N (rR) 1 \| ^2 \nn \\
& = & \lim \ip{p}{x_N (rR) ^* x_N (rR) p} \nn \\
& \leq & \lim \ip{p}{\Phi _r \left( x_N (R) ^* x_N (R) \right) p} \nn \\
& = & \lim \ip{p}{ \Phi _r (I +T _N )  p} \nn \\
& = & \ip{p}{ (I+ \Phi _r (T ) ) p}, \nn \ea by Equation (\ref{formula}), and this proves that 
$$ x(rR) ^* x(rR) \leq I +  \Phi _r (T) , $$ which is in turn bounded above by $I+T_r$ by the discussion preceding the lemma.
\end{proof}

Consider the net $\Delta _r$, $0< \Delta _r \leq I$, for $0<r<1$. Since this net is uniformly bounded, there is a $WOT-$convergent subsequence $\Delta _k := \Delta _{r_k}$ with limit $0 \leq \delta \leq I$. To show that the entire net $\Delta _r$ converges in the weak operator topology to $\delta$, it suffices to show that any $WOT-$convergent subsequence of the net $\Delta _r$ has the same limit. 

\begin{prop} \label{uniquelim}
Let $\Delta _r := (I +T_r)^{-1}$. Then $\Delta _r \stackrel{WOT}{\rightarrow} \Delta$, where $\Delta := (I + T ) ^{-1}$. 
\end{prop}
This proposition does most of the work of proving our NC Fatou Theorem; once it is established it remains only to improve $WOT$ convergence to $SOT$ convergence, which is a routine argument using the resolvent identity. To prove the proposition, we begin with a lemma. 
\begin{lemma}
    Let $\Delta _k := \Delta _{r_k}$ be any $WOT-$convergent subsequence of $\Delta _r$ with limit $\delta$.
Then $0 < \delta \leq I$ is injective so that $\delta ^{-1}$ is a closed, positive operator.
\end{lemma}
\begin{proof}
Suppose that $\delta$ has kernel so that there is a non-zero $h \in \hardy$ so that $\delta h =0$. It follows that
\begin{align*} 0 & =  \ip{h}{\delta h} = \lim _{r_k\uparrow 1} \ip{\Delta _{r_k} h}{h} \\
& = \frac{1}{2} \lim _{k \rightarrow \infty } \ip{g_k}{(I - \nbre B(r_kR) ) ^{-1} g _k }, \end{align*} where 
$g_k := (I - B(r_kR) ^* ) h$. By the NC Schwarz inequality and M\"{o}bius transforms, $B(r_kR)$ is a strict contraction so that $0<(I - \nbre B(r_k R) ) <2I$ is invertible for $0<r_k<1$ \cite[Theorem 2.4]{Pop-freeholo}. (Alternatively, $B(rR)$ is a strict contraction for any $0<r<1$ by the NC maximum modulus principle, see Remark \ref{strictcon}.) By spectral mapping, 
\begin{align*} \frac{1}{2} I & \leq (I - \nbre B(r_kR) ) ^{-1}, 
\intertext{so that} 
 0 & = \ip{h}{\delta h} \\
& = \frac{1}{2} \lim _k \ip{g_k}{(I - \nbre B(r_k R)  ) ^{-1} g _k} \\ 
& \geq  \frac{1}{4} \lim _k \| g _k \| ^2  \\ 
& = \frac{1}{4} \lim _{k} \| (I - B(r_kR) ^* ) h \| ^2  \\
& = \frac{1}{4} \| (I - B(R)^*) h \| ^2. \end{align*} 
Since $B(R)$ is a contraction, it follows also that $(I - B(R)) h =0$. However, since we assume that $B \in \scr{L} _d$ is non-constant and hence strictly contractive in the NC unit row-ball, $I - B(Z)$ must be invertible. Hence, 
\ba 0 & = & U_{\mr{t}} (I - B(R) ) h  = U_{\mr{t}} (I-B(R) ) U_{\mr{t}} ^* U_{\mr{t}} h \nn \\
& = & (I - B(L) ) h^{\mr{t}}, \nn \ea so that for any $Z$, 
$$ 0 = (I - B(Z) ) h^{\mr{t}} (Z). $$ We conclude that $h \equiv 0$, and that $\delta$ is injective.
\end{proof}
The following familiar fact will be used repeatedly in the sequel:
\begin{lemma} \label{weakmultconv}
Let $\cH _{nc} (K)$ be an NC-RKHS on an NC set $\Om \subseteq \C ^d _\N$ with CPNC kernel $K$.
A seqeunce, $(h_j)$, in $\cH _{nc} (K)$ converges weakly to some $h \in \cH _{nc} (K)$ if and only if $(h_j)$ is uniformly norm-bounded and $h_j  \rightarrow h$ pointwise in $\Om$. 

Assume that $\cH _{nc} (K)$ contains the constant function $\mr{id} (Z) := I_n$. A sequence, $(H_j)$, of (left or right) NC multipliers of $\cH _{nc} (K)$ converges in the weak operator topology to a (left or right) multiplier $H$ (in the sense that $M^R _{H_j} \stackrel{WOT}{\rightarrow} M^R _H$ or $M^L _{H_j} \stackrel{WOT}{\rightarrow} M^L _H$) if and only if left or right multiplication by $H_j$, respectively, is uniformly bounded in operator norm and $H_j \rightarrow H$ pointwise in $\Om$. 
\end{lemma}
\begin{proof}
If $h_j \stackrel{w}{\rightarrow} h$, where $w$ denotes weak convergence, then this sequence is bounded in Hilbert space norm and for any $(Z,y,v) \in \Om _n \times \C ^n \times \C ^n$, $$ y^* h_j (Z) v = \ip{K\{Z , y ,v \}}{h_j}_K \rightarrow \ip{K\{Z, y , v \}}{h}_K = y^* h(Z) v.$$ Similarly if $\mr{id} \in \cH _{nc} (K)$ and $M^R _{H_j} \stackrel{WOT}{\rightarrow} M^{R} _H$, then $(M^R _{H_j})$ is uniformly bounded in operator norm, and if $h_j := M^R _{H_j} \cdot \mr{id} \in \cH _{nc} (K)$ then $h_j \stackrel{w}{\rightarrow} h := M^R _H \cdot \mr{id}$ and $h_j (Z) = H_j (Z)$, $h(Z) = H(Z)$. 

Conversely, the linear span of the NC kernels $K\{ Z, y , v \}$ is dense in $\cH _{nc} (K)$ (in fact, linear combinations of NC kernels are NC kernels), so pointwise convergence and uniform boundedness readily implies weak or WOT convergence. 
\end{proof}

\begin{proof}[Proof of Proposition \ref{uniquelim}]
Let $\Delta _k$ be any $WOT-$convergent subsequence with limit $\delta$, it suffices to show that $\delta = \Delta = (I+T) ^{-1}$, where recall that $q_T$ is the closure of $q_{\mu _{ac}}$, and $\mu = \mu _B$. This will prove that every $WOT-$convergent subsequence of the $\Delta _r$ has the same limit, $\Delta$, so that the entire net converges in $WOT$ to $\Delta$. By the previous lemma, $\delta ^{-1}$ is closed, positive, and bounded below by $1$ and we can define $T' := \delta ^{-1} - I$. By spectral mapping, $T'$ is positive semi-definite.

By \cite[Theorem 1.5]{Pop-entropy}, for any $0<r<1$, there is a bounded $y ^{(r)} (R) \in R^\infty _d$ so that  $\Delta _r ^{-1} = I +T _r = y ^{(r)} (R) ^* y ^{(r)} (R)$, where 
$$ y ^{(r)} (R) = c_r ^{-1} (M ^R _{\Delta _r 1} ) ^{-1}; \quad \quad c_r := \ip{1}{\Delta _r 1} ^{-1/2} , $$ so that $y ^{(r)} (R) ^{-1} = c_r M^R _{\Delta _r 1 }$. Here, note that any right outer $y(R) \sim R^\infty _d$ is always pointwise invertible in the NC unit ball $\B ^d _\N$ \cite[Lemma 3.2]{JM-freeSmirnov}. Moreover, 
$$ \Delta _r = (I + T_r ) ^{-1} = y ^{(r)} (R) ^{-1} y ^{(r)} (R) ^{-*}, $$ so that $\| y ^{(r)} (R) ^{-1} \| \leq 1$. 
The Taylor coefficients of $y ^{(r)} (R) ^{-1} $ are then:
$$ \ip{ L^\alpha 1}{y ^{(r)} (R) ^{-1} 1} = c_r \ip{L^\alpha 1} {\Delta _r 1}. $$ 
Consider the operator $c M^R _{\delta 1}$ where $c:= \ip{1}{\delta 1} ^{-\frac{1}{2}}$, initially defined on free polynomials.  Then, since $\Delta _k = \Delta _{r_k} \stackrel{WOT}{\rightarrow} \delta$, we claim $cM^R _{\delta 1}$ defines a contractive right multiplier, $y(R) ^{-1}$ so that $y_k (R) ^{-1} := y ^{(r_k)} (R) ^{-1} $ converges in the weak operator topology to $y(R) ^{-1}$. Indeed, for any $\alpha, \beta \in \F ^d$, 
$$ \ip{ L^{\alpha} 1}{ (c M^R _{\delta 1} - y_k(R) ^{-1} ) L^\beta 1}  = \left\{ \begin{array}{ccc}  \ip{ L^{\ga} 1}{(\delta - \Delta _k ) 1} & \ & \alpha = \beta \ga \\ \ip{1}{ L^\ga (\delta - \Delta _k) 1} =0 & & \beta = \alpha \ga, \ \ga \neq \emptyset \\
0 & & \mbox{else} \end{array} \right\} \quad  
\longrightarrow 0, $$ and so the same holds replacing the free monomials $L^\alpha, L^\beta$ with any free polynomials $p,q$. Since the free polynomials are dense in $\hardy$ and the sequence $y_k (R) ^{-1}$ is uniformly norm bounded, it follows that 
$$ y_k (R) ^{-1} \stackrel{WOT}{\rightarrow} cM^R _{\delta 1}, $$ so that $y(R) ^{-1} := c M^R _{\delta 1}$ is a contractive right multiplier.
We claim that $\delta = y(R ) ^{-1} y(R) ^{-*}$. Indeed this follows because $$ \Delta _k = y_k (R) ^{-1} y_k (R) ^{-*} \stackrel{WOT}{\rightarrow} \delta.$$ For any $(Z,y,v)$ and another $(W,b,c) \in \B ^d _m \times \C ^m \times \C ^m$,
$$ \ip{ K \{ Z , u , v \} }{ \delta  K \{ W , b, c \} }  =  \lim _k \ip{ K \{ Z ,u , v \} }{ \Delta _k K \{ W , b, c \} }. $$ Since $y_k (R) ^{-1} \stackrel{WOT}{\rightarrow} y(R) ^{-1}$, Lemma \ref{weakmultconv} implies that $y_k ^{\mr{t}} (Z) ^{-1}$ converges pointwise to $y ^{\mr{t}} (Z) ^{-1}$ in the NC unit ball,
\ba \ip{ K \{ Z , u , v \} }{ \Delta _k K \{ W , b, c \} } & = & \ip{ (y_k (R) ^{-1} ) ^* K \{ Z, u , v \} }{y_k (R) ^{-*} K \{ W , b , c \} } \nn \\
& = & \ip{ K \{ Z , u , y_k ^{\mr{t}}  (Z) ^{-1} v \} }{K \{ W , b , y_k ^{\mr{t}} (W) ^{-1} c \}} \nn \\
& = & \ipcn{u}{K(Z,W)[  y_k ^{\mr{t}} (Z) ^{-1} vc^* y_k ^{\mr{t}} (W) ^{-*}  ] b } \nn \\
& \rightarrow & \ipcn{u}{K(Z,W)[  y ^{\mr{t}} (Z) ^{-1} vc^* y ^{\mr{t}} (W) ^{-*}  ] b } \nn \\
& = & \ip{ K \{ Z, u , v \} }{ y (R) ^{-1} y (R) ^{-*} K \{ W , b , c \} }, \nn \ea  and the claim follows. Consider $y_k (R) := y  ^{(r_k)} ( R)$. Observe that the sequence $y_k ^\mrt= y_k (R) 1$ is uniformly bounded in $\hardy$: 
\ba \| y_k (R) 1 \| ^2 & = & \ip{1}{(I +\nbre H_B (r_k R ) ) 1} _{\hardy} \nn \\
& = & 1 + \mu _{B _{r_k}} (I) = 1  + \nbre H_{B_{r_k} ; \emptyset}  \nn \\
& = & 1 + \nbre H_B (0) = 1 + \mu _B (I) < + \infty. \nn \ea 
It follows that there is a weakly convergent subsequence $y_j ^\mrt = y_{k_j} ^\mrt $ with limit $\wt{y} ^\mrt \in \hardy$. However, we already know that for any $Z \in \B ^d _\N$, $y_k ^{\mr{t}} (Z) ^{-1}$ converges to $y^{\mr{t}} (Z) ^{-1}$ so that $y_k ^{\mr{t}} (Z)$ converges to $y^{\mr{t}} (Z)$ pointwise in $\B ^d _\N$. By Lemma \ref{weakmultconv}, it follows that the subsequence $y_j ^\mrt = y_{k_j} ^\mrt$ converges pointwise to both $y ^\mrt$ and $\wt{y} ^\mrt $ so that $y ^\mrt = \wt{y} ^\mrt \in \hardy$. In particular, for any free polynomial, $p$, 
\ba \| y(R) p \| ^2 & = & \| p(L) y ^\mrt \| ^2 \nn \\
& = & \lim _k \ip{p(L) y_k ^\mrt  }{p(L) y ^\mrt } \nn \\
& \leq & \| y(R) p \| \limsup _k \| p(L) y_k ^\mrt  \| , \nn \ea 
and it follows that 
\ba \| y (R) p \| ^2  & \leq & \limsup _k \| p(L) y_k ^\mrt  \| ^2 \nn \\
& = & \limsup _k \ip{p}{y_k (R) ^* y_k (R) p}  \nn \\
& = & \limsup _k  \ip{p}{(I + \nbre H _B (r_k R ) )p}  \nn \\
& = & \lim _k \mu _{B_{r_k}} (p^* p)  + m (p^* p )  \nn \\
& = & (\mu _B +m ) (p^* p), \quad \quad \mbox{by Lemma \ref{rapprox}.} \nn \ea 
Since sums of squares, $p^*p$, of free polynomials $p \in \fp$ are norm-dense in the positive cone of $\scr{A} _d$, this proves that the vector state $m_{y^\mrt} (L^\alpha ) := \ip{y ^\mrt }{L^\alpha y ^\mrt } _{\hardy}$ is bounded above by $\mu +m = \mu _B +m$ \cite[Lemma 4.6]{JM-freeAC}. However, as proven above, $\delta = y(R) ^{-1} y(R) ^{-*}$, so that by definition $I +T ' = \delta ^{-1} = y(R) ^* y(R)$. Hence,
\ba q_y (a_1 ,a_2 ) &:= &  m_{y^\mrt} (a_1 ^* a_2 ); \quad \quad a_1, a_2 \in \A  \nn \\
& = & q_{I+T'} (a_1, a_2). \nn \ea  That is, given $\dom{q_{\mu_B} + q_m} = \A$, 
$$ q_{T'} + q_m = q_{I+T'} = q_y  \leq q_{\mu _B + m} = q _{\mu _B} + q_m, $$ in the sense of the partial order on positive semi-definite quadratic forms, see Equation (\ref{formpo}). This proves that $q_{T'}$ is a closable positive semi-definite quadratic form obeying $q_{T'} \leq q_{\mu _B}$. By Theorem \ref{NCLD} the closure of $\ov{ q_{\mu _{ac}} } $ is $q _{T}$, and $q_{\mu_{ac}}$ is the maximal closable positive semi-definite quadratic form bounded above by $q_\mu$. Maximality then implies that $T' \leq T$ in the sense that $q_{T'} \leq q_T$. That is, $\mr{Dom} \sqrt{T} \subseteq \mr{Dom} \sqrt{T'}$, and $$ \| \sqrt{T '} h \| ^2  = q_{T '} (h, h) \leq q_{T} (h,h) = \| \sqrt{T} h \| ^2, $$  for all $h \in \mr{Dom} \sqrt{T}$.

Conversely, for any $0<r<1$, the harmonic majorant inequality of Lemma \ref{subharlem} implies that \be x (rR) ^* x(rR) \leq I +T _r.  \label{harlemr} \ee (Recall that $x(R) ^* x(R) = I +T$.) 
By \cite[Chapter VI, Theorem 2.21]{Kato} (see also \cite[Proposition 1.1]{Simon1}), closed positive semi-definite operators $T_1, T_2$ obey $q_{T_1} \leq q_{T_2}$ if and only if $(\la I + T_2) ^{-1} \leq (\la I +T_1 ) ^{-1}$ for any $\la >0$. It follows that the above inequality (\ref{harlemr}) is equivalent to:
$$ (I + T_r ) ^{-1}  = \Delta _r \leq x(rR) ^{-1} x(rR) ^{-*}. $$
In the above, recall that $x(R) ^{-1}$ and hence $x(rR) ^{-1}$ are contractions. In particular, for each $r = r_k$, 
$$ \underbrace{\Delta _{r_k}}_{\stackrel{WOT}{\rightarrow} (I+T' ) ^{-1}}   \leq \underbrace{x(r_kR) ^{-1} x(r_kR) ^{-*} }_{\stackrel{SOT}{\rightarrow} (I +T) ^{-1}},$$ where the $SOT$ convergence of $x(rR) ^{-1} x(rR) ^{-*}$ to $(I+T) ^{-1}$ follows from Lemma \ref{SRlimlem}.  Again by \cite[Chapter VI, Theorem 2.21]{Kato}, we conclude that 
$$ I+ T  \leq I +T' , $$ so that $q_T \leq q_{T'}$ and hence $q_T = q_{T'}$. By the uniqueness of the Riesz representation of closed, densely-defined positive semi-definite quadratic forms, $T=T'$ \cite[Chapter VI, Theorem 2.1, Theorem 2.23]{Kato}.  
\end{proof}

\begin{remark}
Observe that with $x(R) ^* x(R) = I + T$, and $y ^{(r)} (R) ^* y ^{(r)} (R) = I + T _r$, we have that 
$x (rR) 1 \rightarrow x(R)1 = x ^\mrt$ in $\hardy$, and similarly since $(I +T_r ) ^{-1} $ converges in $WOT$ to $(I +T ) ^{-1}$, one can argue that $y^{(r) ; \mrt} := y ^{(r)} (R) 1$ converges weakly to $x(R)1$ in $\hardy$. (That $y^{(r) ; \mrt} := y ^{(r)} (R) 1$ converges weakly to $x(R)1$ in $\hardy$ was shown in the proof of Proposition \ref{uniquelim}.) However, $y ^{(r); \mrt}$ generally does not converge in norm to $x ^\mrt$, as this would imply that $\mu = \mu _{ac}$, which is generally not true. Indeed, if $y^{(r); \mrt} \stackrel{\bH ^2}{\rightarrow} x^\mrt$, then for any $a_1, a_2 \in \A$, 
\begin{align*}  \underbrace{\mu_r (a_1^* a_2) + m (a_1 ^* a_2)}_{\verteq} & \rightarrow \mu (a_1^* a_2 ) + m (a_1 ^* a_2) \quad \quad \mbox{by Lemma \ref{rapprox}} \\
m_{y^{(r) ; \mrt}} (a_1 ^* a_2 ) & \rightarrow m_{x^\mrt} (a_1 ^* a_2 ) = \mu _{ac} (a_1 ^* a_2) + m (a_1 ^* a_2). \end{align*}

\end{remark}
\begin{remark}
Since $T' = T$, Equation (\ref{NCsubharm}) becomes: 
\be x(rR) ^* x(rR) \leq I + T_r = 2 (I - B(rR)^* ) ^{-1} (I - \nbre B(rR) ) (I - B(rR) ) ^{-1}. \ee Equivalently, setting $X(rR) := x(rR) (I - B(rR))$, 
\be X(rR) ^* X(rR) \leq 2 (I - \nbre B(rR) ). \ee
This shows that the free harmonic function $2(I-\nbre B^{\mr{t}} (Z))$ is a harmonic majorant for the free pluri-subharmonic function $X^{\mr{t}} (Z) ^* X ^{\mr{t}} (Z)$. 
\end{remark}

We are finally in position to prove our main theorem. 

\begin{proof}[Proof of NC Fatou Theorem]
We have proven that $\Delta _r := (I +T _r ) ^{-1}$ converges to $\Delta := (I +T) ^{-1}$ in $WOT$ and a similar analysis for $\Delta _r (s) = (sI + T_r ) ^{-1}$, $s>0$ shows that $\Delta _r (s)$ converges $WOT$ to $\Delta (s) = (sI +T ) ^{-1}$.

Using $WOT$ convergence of the $\Delta _r (s)  = (sI + T_r ) ^{-1}$ to $\Delta (s)$ for $1\leq s \leq 2$ implies, by the resolvent formula,
$$ \frac{1}{\eps} \left( (I + T_r ) ^{-1} - ((1+\eps ) I + T_r ) ^{-1} \right) = (I+T_r ) ^{-1} ((1+\eps ) I + T_r )^{-1} , $$ is WOT-convergent to 
$$ (I +T ) ^{-1} ((1+\eps ) I +T ) ^{-1}, $$ for any $\eps \in [0,1]$. Since $(I + T _r ) ^{-1} ( (1 +\eps ) I + T _r ) ^{-1}$ is uniformly bounded for $\eps \in [0,1]$ and $0 < r <1$, taking the limit as $\eps \downarrow 0$ shows that 
$$ (I +T _r ) ^{-2} \stackrel{WOT}{\rightarrow} (I +T ) ^{-2}, $$ is $WOT-$convergent, and this implies  $SOT-$convergence of $(I +T _r ) ^{-1}$ since 
$$ \| (I + T_r ) ^{-1} h \| ^2 = \ip{h}{(I+T_r) ^{-2} h}. $$ Hence $T_r \rightarrow T$ in the strong resolvent sense. 
\end{proof} 
Having proved that $T$ is recovered from the $T_r$ in the sense of strong resolvent convergence, we consider the problem of exhibiting $T$ more explicitly. Ideally, one would like to prove that 
\[
T=(I-B(R)^*)^{-1} (I-B(R)^*B(R))(I-B(R) )^{-1}
\]
(suitably interpreted). One way of making this precise would be to claim that $\ran{I-B(R)}$ belongs to $\mr{Dom} \, \sqrt{T}$ and that for all $f,g\in \hardy$ 
\[
\langle T(I-B(R)f, (I-B(R))g\rangle = \langle (I-B(R)^*B(R))f, g\rangle.
\]
If this were true unrestrictedly, it would prove, for example, that $T=0$ (equivalently, $\mu_{ac}=0$, that is, $\mu$ is singular) if and only if $B(R)$ is an isometry (that is $B^\mr{t}$ is a right NC inner function). This correspondence--that a function $b$ is inner if and only if its Clark measure $\mu_b$ is singular--of course holds in one variable, via the standard form of Fatou's theorem as described in the introduction. Here, at present, our results are somewhat less satisfactory, but we are able to prove the desired inequality in one direction:

\begin{thm} \label{FatouForm}
$\mr{Dom} \sqrt{I+T}$ contains $\ran{I-B(R)}$ and
\be (I - B(R) ^* ) (I +T) (I - B(R) ) \leq 2 (I -\nbre B(R) ). \ee 
\end{thm}

\begin{proof}
By Theorem \ref{tfactor} we have that $x(R) ^* x(R) = I +T $, $x(R) = U ^* \sqrt{I +T}$ (by construction and also by polar decomposition) so that $\nbdom x(R) = \mr{Dom} \, \sqrt{I+T}$. We claim that $\ran{I - B(R)} \subseteq \mr{Dom} \, \sqrt{I+T} = \nbdom x(R)$. First define, for each $0<r<1$, $X(rR) := x(rR) (I - B(rR) )$. Then by Lemma \ref{subharlem},
\ba X(rR) ^* X(rR) & = & (I - B(rR) ^* ) x(rR)^* x(rR) (I - B(rR) ) \nn \\
& \leq & (I - B(rR) ^* ) (I +T _r ) (I - B(rR) ) \nn \\
& = & 2 (I - \nbre B(rR) ). \label{harmajor1} \ea 
It follows that $X(rR)$ is uniformly norm-bounded for $0<r<1$. Let $X_k (R) = X(r_k R)$ be any $WOT-$convergent subsequence with limit $\wt{X} (R)$. Then $X ^{\mr{t}} (r_k Z)$ necessarily converges pointwise to $\wt{X} ^\mr{t} (Z)$ for any $Z \in \B ^d _\N$ by Lemma \ref{weakmultconv}. However, we also have 
$$ X^{\mr{t}} (r_k Z) = (I - B^{\mr{t}} (r_k Z) ) x^{\mr{t}} (r_k Z ) \rightarrow X^{\mr{t}} (Z) := (I - B^{\mr{t}} (Z) ) x^{\mr{t}} (Z). $$ This proves that any such $WOT$ limit is unique and equal to $X(R) = x(R) (I - B(R) )$, so that, in particular, $\ran{I - B(R)} \subseteq \nbdom x(R)$, and $X(rR) \stackrel{WOT}{\rightarrow} X(R)$. Moreover, by Lemma \ref{sotstar}, since $X(rR)$ is uniformly norm bounded, $X(rR) \stackrel{SOT-*}{\rightarrow} X(R)$. Taking the limit of Equation (\ref{harmajor1}) then yields:
$$ X(R) ^* X(R) \leq 2 (I - \nbre B(R) ). $$
\end{proof}

\begin{cor} \label{innercor}
    Let $\mu =  \mu _B$ be the NC Clark measure of $B \in \scr{L} _d = [ \bH ^\infty _d ] _1$. If $B$ is inner (that is, $B(R) = M^R _{B ^{\mr{t}}}$ is an isometry) then $\mu$ is singular. 
\end{cor}

\begin{proof}
Recall that $B$ is inner if and only if $B(L)$, or equivalently $B(R) = U_{\mr{t}} B(L) U_{\mr{t}}$ are isometries. We have that the closure of $q_{\mu _{ac}}$ is  $q_T$ where $T$ is the strong resolvent limit of the $T_r = \nbre H_B (rR)$ if $\mu = \mu _B$. 
In particular for any $h := (I - B(R) ) g$, if $B$ is inner then:
\ba q_{(I+T)} (h,h) &\leq & 2\ip{g}{(I -\nbre B(R)) g}  \nn \\
& = & \ip{(I-B(R))g}{g} + \ip{g}{(I-B(R)) g} \nn \\
& = & \ip{(I-B(R)) g}{(I-B(R)) g} \quad \quad \quad \mbox{Using that } B(R) ^* B(R) = I \nn \\
& = & \ip{h}{h} = q_I (h,h). \nn \ea 
This proves that $I+T \leq  I$ so that $T \equiv 0$.
\end{proof}

Inspecting the proof of Theorem~\ref{FatouForm}, it is not too hard to show that we would obtain equality (rather than just the one inequality) if we knew that for all $B(R)$ we could factor
\begin{equation}\label{eqn:factor 1-ReB}
2I-B(R)^*-B(R) =G(R)^*G(R).
\end{equation}
for some $G \in R^\infty _d$. 
Indeed, if this is so then for all $r<1$ we would have
\[
(I-B(rR)^*)(I+T_r)(I-B(rR)) \geq G(rR)^*G(rR),
\]
and by taking $SOT$ limits
\[
(I-B(R)^*)x(R)^*x(R)(I-B(R)) \geq G(R)^*G(R) = 2(I- \nbre B(R)),
\]
as desired. In one variable the factorization (\ref{eqn:factor 1-ReB}) always holds, indeed on the circle we have 
\[
2-\overline{b(\zeta)} -b(\zeta) = |1-b(\zeta)|^2+ 1-|b(\zeta)|^2 \geq |1-b(\zeta)|^2,
\]
so that $2-\overline{b(\zeta)} -b(\zeta)$ is log-integrable, and hence there is an outer function $g$ with $2-\overline{b(\zeta)} -b(\zeta) =|g(\zeta)|^2$. Thus our results allow us to fully recover the known form of the Radon-Nikodym derivative in the one-variable case.


\begin{thebibliography}{10}

\bibitem{Fatou}
P.~Fatou.
\newblock S\'{e}ries trigonom\'{e}triques et s\'{e}ries de {T}aylor.
\newblock {\em Acta. Math.}, 37:335--400, 1906.

\bibitem{Clark1972}
D.N. Clark.
\newblock One dimensional perturbations of restricted shifts.
\newblock {\em J. Anal. Math.}, 25:169--191, 1972.

\bibitem{Aleks2}
A.~B. Aleksandrov.
\newblock Multiplicity of boundary values of inner functions. ({R}ussian).
\newblock {\em Izv. Akad. Nauk Arm. SSR}, 22:490--503, 1987.

\bibitem{Aleks1}
A.B. Aleksandrov.
\newblock On the existence of nontangential boundary values of
  pseudocontinuable functions. ({R}ussian).
\newblock {\em Zapiski Nauchnykh Seminarov POMI}, 222:5--17, 1995.

\bibitem{Saks}
E. Saksman.
\newblock An elementary introduction to {C}lark measures.
\newblock In {\em Topics in Complex Analysis and Operator Theory}, pages
  85--136. Servicio de Publicaciones de la Universidad de M{\'a}laga, 2007.

\bibitem{Hoff}
K.~Hoffman.
\newblock {\em Banach spaces of analytic functions}.
\newblock Prentice-Hall, Inc. 1962.

\bibitem{Pop-freeharm}
G.~Popescu.
\newblock Noncommutative transforms and free pluriharmonic functions.
\newblock {\em Advances in Mathematics}, 220:831--893, 2009.

\bibitem{JM-freeAC}
M.T. Jury and R.T.W. Martin.
\newblock Non-commutative {C}lark measures for the {F}ree and {A}belian
  {T}oeplitz algebras.
\newblock {\em J. Math. Anal. Appl.}, 456:1062--1100, 2017.

\bibitem{JM-freeCE}
M.T. Jury and R.T.W. Martin.
\newblock Column-extreme multipliers of the {F}ree {H}ardy space.
\newblock {\em J. Lond. Math. Soc.}, 101:457--489, 2020.

\bibitem{Pop-freeholo}
G.~Popescu.
\newblock Free holomorphic functions on the unit ball of {$B (\mathcal{H})
  ^n$}.
\newblock {\em J. Funct. Anal.}, 241:268--333, 2006.

\bibitem{Pop-freeholo2}
G.~Popescu.
\newblock Free holomorphic functions on the unit ball of {$\mathcal{B}
  (\mathcal{H}) ^n$}, {I}{I}.
\newblock {\em J. Funct. Anal.}, 258:1513--1578, 2010.

\bibitem{BMV}
J.A. Ball, G.~Marx, and V.~Vinnikov.
\newblock Noncommutative reproducing kernel {H}ilbert spaces.
\newblock {\em J. Funct. Anal.}, 271:1844--1920, 2016.

\bibitem{VinPop}
M.~Popa and V.~Vinnikov.
\newblock {$H^2$} spaces of non-commutative functions.
\newblock {\em Complex Analysis and Operator Theory}, 12:945--967, 2018.

\bibitem{KVV}
D.S. Kaliuzhny{\u{\i}}-Verbovetskyi{\u{\i}} and V.~Vinnikov.
\newblock {\em Foundations of free noncommutative function theory}, volume 199.
\newblock American Mathematical Society, 2014.

\bibitem{Nik-shift}
N.~K. Nikolskii.
\newblock {\em Treatise on the shift operator: spectral function theory}.
\newblock Springer-Verlag, 1986.

\bibitem{NF}
B.~Sz.-Nagy and C.~Foia\c{s}.
\newblock {\em Harmonic analysis of operators on \uppercase{H}ilbert space}.
\newblock American Elsevier publishing company, Inc., New York, N.Y., 1970.

\bibitem{Kosaki}
H.~Kosaki.
\newblock Lebesgue decomposition of states on a von {N}eumann algebra.
\newblock {\em Amer. J. Math.}, 107:697--735, 1985.

\bibitem{Takesaki}
M.~Takesaki.
\newblock On the conjugate space of operator algebra.
\newblock {\em Tohoku Math. J. (2)}, 10:194--203, 1958.

\bibitem{Connes}
A. Connes.
\newblock {\em Noncommutative geometry}.
\newblock Springer, 1994.

\bibitem{Simon1}
B.~Simon.
\newblock A canonical decomposition for quadratic forms with applications to
  monotone convergence theorems.
\newblock {\em J. Funct. Anal.}, 28, 1978.

\bibitem{vN3}
J.~von Neumann.
\newblock On rings of operators {I}{I}{I}.
\newblock {\em Ann. Math.}, 41:94--161, 1940.

\bibitem{Pop-entropy}
G.~Popescu.
\newblock {\em Entropy and multivariable interpolation}.
\newblock American Mathematical Society, 2006.

\bibitem{Upbook}
H.~Upmeier.
\newblock {\em Toeplitz operators and index theory in several complex
  variables}, volume~81.
\newblock Birkh{\"a}user, 2012.

\bibitem{SSS}
G.~Salomon, O.~Shalit, and E.~Shamovich.
\newblock Algebras of bounded noncommutative analytic functions on subvarieties
  of the noncommutative unit ball.
\newblock {\em Trans. Amer. Math. Soc.}, 370:8639--8690, 2018.

\bibitem{Voic2}
D.V. Voiculescu.
\newblock Free analysis questions {I}{I}: the {G}rassmannian completion and the
  series expansions at the origin.
\newblock {\em J. Reine Angew. Math.}, 645:155--236, 2010.

\bibitem{JMS-NCrat}
M.T. Jury, R.T.W. Martin, and E.~Shamovich.
\newblock Non-commutative rational functions in the full {F}ock space.
\newblock {\em Trans. Amer. Math. Soc.}, https://doi.org/10.1090/tran/8418,
  2021.

\bibitem{Pop-dil}
G.~Popescu.
\newblock Isometric dilations for infinite sequences of noncommuting operators.
\newblock {\em Trans. Amer. Math. Soc.}, 316:523--536, 1989.

\bibitem{Ag-Mc}
J.~Agler and J.E. McCarthy.
\newblock Global holomorphic functions in several non-commuting variables.
\newblock {\em Canad. J. Math.}, 67:241--285, 2015.

\bibitem{Taylor}
J.L. Taylor.
\newblock A general framework for a multi-operator functional calculus.
\newblock {\em Adv. Math.}, 9:183--252, 1972.

\bibitem{Taylor2}
J.L. Taylor.
\newblock Functions of several noncommuting variables.
\newblock {\em Bull. Amer. Math. Soc.}, 79:1--34, 1973.

\bibitem{Voic}
D.V. Voiculescu.
\newblock Free analysis questions {I}: Duality transform for the coalgebra of
  {$\partial _{X: B}$}.
\newblock {\em Int. Math. Res. Not.}, 16:793--822, 2004.

\bibitem{Tak-dual}
M.~Takesaki.
\newblock A duality in the representation theory of {$C^*-$}algebras.
\newblock {\em Ann. Math.}, 85:370--382, 1967.

\bibitem{AgMcY}
J.~Agler, J.E. McCarthy, and N.J. Young.
\newblock Non-commutative manifolds, the free square root and symmetric
  functions in two non-commuting variables.
\newblock {\em Trans. Lond. Math. Soc.}, 5:132--183, 2018.

\bibitem{Aron-RKHS}
N.~Aronszajn.
\newblock Theory of reproducing kernels.
\newblock {\em Trans. Amer. Math. Soc.}, 68:337--404, 1950.

\bibitem{Paulsen-rkhs}
V.~Paulsen and M.~Raghupathi.
\newblock {\em An Introduction to the theory of reproducing kernel {H}ilbert
  spaces}.
\newblock Cambridge Studies in Advanced Mathematics, 2016.

\bibitem{DP-inv}
K.R. Davidson and D.R. Pitts.
\newblock Invariant subspaces and hyper-reflexivity for free semigroup
  algebras.
\newblock {\em Proc. Lond. Math. Soc.}, 78:401--430, 1999.

\bibitem{Ball-Fock}
J.A. Ball, V.~Bolotnikov, and Q.~Fang.
\newblock Schur-class multipliers on the {F}ock space: de {B}ranges-{R}ovnyak
  reproducing kernel spaces and transfer-function realizations.
\newblock In {\em Operator Theory, Structured Matrices, and Dilations, Theta
  Series Adv. Math., Tiberiu Constantinescu Memorial Volume}, volume~7, pages
  101--130. 2007.

\bibitem{JM-ncld}
M.T. Jury and R.T.W. Martin.
\newblock Lebesgue decomposition of non-commutative measures.
\newblock {\em Int. Math. Res. Not.}, In press, 2020.

\bibitem{DK-dilation}
K.R. Davidson and E.G. Katsoulis.
\newblock Dilation theory, commutant lifting and semicrossed products.
\newblock {\em Doc. Math.}, 16:781--868, 2011.

\bibitem{Cuntz}
J.~Cuntz.
\newblock Simple {$C^*-$}algebras generated by isometries.
\newblock {\em Comm. Math. Phys.}, 57:173--185, 1977.

\bibitem{JM-freeSmirnov}
M.T. Jury and R.T.W. Martin.
\newblock Operators affiliated to the free shift on the free {H}ardy space.
\newblock {\em J. Funct. Anal.}, 277:108285, 2019.

\bibitem{JM-F2Smirnov}
M.T. Jury and R.T.W. Martin.
\newblock The {S}mirnov classes for the {F}ock space and complete {P}ick spaces.
\newblock {\em Indiana Univ. Math. J.}, 70:269--284, 2021.

\bibitem{Kato}
T.~Kato.
\newblock {\em Perturbation theory for linear operators}.
\newblock Springer, 1976.

\bibitem{RnS1}
M.~Reed and B.~Simon.
\newblock {\em Methods of Modern Mathematical Physics vol. {$1$}, Functional
  Analysis}.
\newblock Academic Press, San Diego, CA, 1980.

\bibitem{BrownHalmos}
P.R. Halmos and A.~Brown.
\newblock Algebraic properties of {T}oeplitz operators.
\newblock {\em J. Reine Angew. Math.}, 213:89--102, 1963.

\bibitem{Pop-vN}
G.~Popescu.
\newblock Von {N}eumann inequality for {$(B (\mathcal{H}) ^n ) _1$}.
\newblock {\em Math. Scand.}, 68:292--304, 1991.

\bibitem{Sha2013}
O.M. Shalit.
\newblock Operator theory and function theory in {D}rury--{A}rveson space and
  its quotients.
\newblock In {\em Handbook of Operator Theory}, pages 1125--1180. Springer,
  2015.

\bibitem{DLP-ncld}
K.R. Davidson, J.~Li, and D.R. Pitts.
\newblock Absolutely continuous representations and a {K}aplansky density
  theorem for free semigroup algebras.
\newblock {\em J. Funct. Anal.}, 224:160--191, 2005.

\bibitem{analysisnow}
G.K. Pedersen.
\newblock {\em Analysis now}.
\newblock Springer, 2012.

\bibitem{Pop-NCdisk}
G.~Popescu.
\newblock Non-commutative disc algebras and their representations.
\newblock {\em Proc. Amer. Math. Soc.}, 124:2137--2148, 1996.

\bibitem{Paulsen}
V.~Paulsen.
\newblock {\em Completely Bounded Maps and Operator Algebras}.
\newblock Cambridge University Press, New York, NY, 2002.

\end{thebibliography}

\end{document}